\documentclass[11pt]{articlefederico2017}
\usepackage{graphicx}
\usepackage{amsfonts }
\usepackage{amsmath}
\usepackage{fullpage}
\usepackage{amssymb}
\usepackage{amsthm}
\usepackage{tikz}
\usepackage{lipsum}
\usepackage{mathrsfs}
\usepackage{hyperref}
\usepackage{float}
\usepackage{multicol,cleveref}

\usepackage{xcolor}
\newcommand{\gris}[1]{\textcolor{darkslategray}{#1}}
 \definecolor{darkslategray}{rgb}{0.18, 0.31, 0.31}

\definecolor{ganador}{HTML}{D8345F}

\usepackage{subcaption}

\theoremstyle{theorem}
\newtheorem{theorem}{Theorem}[section]

\newtheorem{lemma}[theorem]{Lemma}

\newtheorem{proposition}[theorem]{Proposition}

\theoremstyle{definition}
\newtheorem{definition}[theorem]{Definition}

\newtheorem{example}[theorem]{Example}
\newtheorem{remark}[theorem]{Remark}

\makeatletter
\newtheorem*{rep@theorem}{\rep@title}
\newcommand{\newreptheorem}[2]{%
\newenvironment{rep#1}[1]{%
 \def\rep@title{#2 \ref{##1}}%
 \begin{rep@theorem}}%
 {\end{rep@theorem}}}
\makeatother
\newreptheorem{theorem}{Theorem}

\newcommand{\C}{{\mathbb{C}}}
\newcommand{\Z}{{\mathbb{Z}}}
\newcommand{\R}{{\mathbb{R}}}
\newcommand{\E}{\mathcal{E}}

\newcommand{\I}{\mathcal{I}}
\newcommand{\J}{\mathcal{J}}
\renewcommand{\L}{\mathcal{L}}

\newcommand{\cH}{\mathcal{H}}

\newcommand{\bi}{\bar{i}}
\newcommand{\bj}{\bar{j}}

\newcommand{\conv}{\mathsf{conv}}
\renewcommand{\int}{\mathsf{int}}

\renewcommand{\deg}{\mathsf{deg}}
\renewcommand{\dim}{\mathsf{dim}}
\newcommand{\mv}{\mathsf{MV}}
\newcommand{\vol}{\mathsf{Vol}}

\newcommand{\e}{{\mathsf{e}}}
\newcommand{\f}{{\mathsf{f}}}

\newcommand{\B}{{\mathsf{B}}}
\newcommand{\M}{{\mathsf{M}}}
\newcommand{\N}{{\mathsf{N}}}
\newcommand{\T}{{\mathsf{T}}}
\newcommand{\nbr}{\mathsf{nbr}}

\newcommand{\NN}{{\mathcal{N}}}

\newcommand\numberthis{\addtocounter{equation}{1}\tag{\theequation}}

\begin{document}

\title{The harmonic polytope}

\author{\textsf{Federico Ardila\footnote{\textsf{San Francisco State University, Universdidad de Los Andes; federico@sfsu.edu. Partially supported by NSF grant DMS-1855610 and Simons Fellowship 613384.}}\quad and\quad Laura Escobar\footnote{\textsf{Washington University in St. Louis; laurae@wustl.edu. Partially supported by NSF grant DMS-1855598. }}}}

\date{}

\maketitle

%
%
\begin{abstract} 
We study the harmonic polytope, which arose in Ardila, Denham, and Huh's work on the Lagrangian geometry of matroids. We describe its combinatorial structure, showing that it is a $(2n-2)$-dimensional polytope with $(n!)^2(1+\frac12+\cdots+\frac1n)$ vertices and $3^n-3$ facets. We also give a formula for its volume: it is a weighted sum of the degrees of the projective varieties of all the toric ideals of connected bipartite graphs with $n$ edges; or equivalently, a weighted sum of the lattice point counts of all the corresponding trimmed generalized permutahedra. 
\end{abstract}
%
%
%
%

%

\section{Introduction}

Motivated by the Lagrangian geometry of conormal varieties in classical algebraic geometry, 
Ardila, Denham, and Huh \cite{ADH1} introduced the \emph{conormal fan} $\Sigma_{\M, \M^\perp}$ of a matroid $\M$ -- a Lagrangian analog of the better known Bergman fan $\Sigma_\M$ \cite{AK}. 
They used the conormal fan $\Sigma_{\M, \M^\perp}$ to give new geometric interpretations of the Chern-Schartz-MacPherson cycle of a matroid --  introduced by L\'opez de Medrano, Rinc\'on, and Shaw in \cite{LRS} -- and of the $h$-vectors of the broken circuit complex $BC(\M)$ and independence complex $I(\M)$ of $\M$. This geometric framework allowed them to prove that these vectors are log-concave,  as conjectured by Brylawski and Dawson \cite{Brylawski, Dawson} in the 1980s. 

In their work, 
Ardila, Denham, and Huh encountered two polytopes associated to a positive integer $n$: the \emph{harmonic polytope} $H_{n,n}$ and the \emph{bipermutohedron} $\Pi_{n,n}$; the first is a Minkowski summand of the second. Their geometric origin is explained in Section \ref{sec:matroids}.
This paper studies the harmonic polytope $H_{n,n}$; its name derives from the fact that its number of vertices is $(n!)^2H_n$ where $H_n = 1 + \frac12 + \cdots + \frac1n$ is the $n$th harmonic sum.
The harmonic polytope has nice vertex and inequality descriptions, shown in 
Propositions \ref{prop:vertices} and \ref{prop:facets}. We also give a combinatorial formula for its $f$-vector; we note that giving such a description for an arbitrary polytope is $\#P$-hard. \cite{Dyer, Linial}

Computing the volume of an arbitrary polytope is a very difficult task \cite{difficult-volume}. In principle, one could compute the volume of a given polytope by constructing a triangulation and adding the volumes of each of the maximal dimensional simplices. 
In practice, 
this is not a feasible approach: 
Dyer and Frieze showed  that the problem of finding the volume of a polytope is $\#$P-hard  \cite{Dyer-Frieze}. Even among polytopes with well-understood face structures, few exact volume formulas are known.

Our main result is Theorem~\ref{thm:volume}, which computes the volume of $H_{n,n}$.
We use the combinatorial structure of the harmonic polytope $H_{n,n}$ to show that its volume 
is a weighted sum of the degrees of the toric ideals of all bipartite multigraphs on $n$ edges; or equivalently, of the lattice point counts of all the corresponding trimmed generalized permutahedra.

\begin{theorem} \label{thm:volume}
The normalized volume of the harmonic polytope is
\begin{eqnarray*}
\vol(H_{n,n}) 
&=& \sum_{\Gamma} \frac{i(P_\Gamma^-)}{(v(\Gamma)-2)!} \prod_{v\in V(\Gamma)}\deg(v)^{\deg(v)-2}  \\
&=& \sum_{\Gamma} \frac{\deg(X_\Gamma)}{(v(\Gamma)-2)!} \prod_{v\in V(\Gamma)}\deg(v)^{\deg(v)-2}, 
\end{eqnarray*}
summing over all connected bipartite multigraphs $\Gamma$ on edge set $[n]$. 
Here $i(P_\Gamma^-)$ is the number of lattice points in the trimmed generalized permutahedron $P_\Gamma^-$ of $\Gamma$, $X_\Gamma$ is the projective embedding of the toric variety of $\Gamma$
given by the toric ideal of $\Gamma$, $V(\Gamma)$ is the set of vertices of $\Gamma$, and $v(\Gamma) := |V(\Gamma)|$. 
\end{theorem}

To prove Theorem \ref{thm:volume} we observe that the harmonic polytope can be expressed as a Minkowski sum of simplices, so its volume is a sum of the associated mixed volumes. Following an idea of Postnikov \cite{Postnikov}, each mixed volume equals the number of isolated solutions of a system of polynomial equations by Bernstein-Khovanskii-Kushnirenko's theorem, and we can try to count those solutions. In Postnikov's case this is easy because one obtains a system of linear equations, which has 0 or 1 solution.
Our setting is much more subtle because our equations are not linear. To count their common solutions, we establish a connection with the theory of \emph{toric edge ideals} \cite{SVV,Villarreal}. This connection allows us to express each mixed volume in terms of the degree of a toric ideal, the volume of an \emph{edge polytope}, or the number of lattice points of a \emph{trimmed generalized permutahedron}.

In order to get an approximation for the volume of the harmonic polytope, it is desirable to count the non-zero terms in the sum of Theorem \ref{thm:volume}.  
We show that the non-zero mixed volumes are in bijection with the pairs of forests on $[n]$ whose union is connected. We count them in Proposition \ref{prop:summands} by computing in the M\"obius algebra of the partition lattice.

\section{Motivation: the Lagrangian geometry of matroids} \label{sec:matroids}

This section, which is logically independent from the rest of the paper, provides the geometric motivation for this project; it assumes some familiarity with the geometry of matroids. Our discussion overlaps with \cite[Section 2.8]{ADH1}; for further details we refer the reader to \cite{AHK, A, ADH1}. 

The harmonic polytope and the bipermutahedron arose naturally in Ardila, Denham, and Huh's construction \cite{ADH1} of the \emph{conormal fan of a matroid}. They used the bipermutahedron to provide a combinatorial model for the Lagrangian geometry of matroids, and derive interesting combinatorial consequences. Our goal in this section is to explain that the harmonic polytope is the universal polytope that is contained in all such models.

\subsection{Combinatorial Hodge theory and log-concavity for matroids}

\noindent 
\textsf{\textbf{The Chow ring, the Bergman fan, and $f$-vectors.}}
The story begins with the proof 
by Huh \cite{Huh}, Huh-Katz \cite{HK},  and Adiprasito-Huh-Katz \cite{AHK} 
of a series of conjectures by Rota, Heron, Mason, and Welsh in the 1970s and 1980s. Their strongest result is 
that the $f$-vector $f_0, f_1 \ldots, f_{r-1}$ of the broken circuit complex of a matroid $M$ is log-concave. 

When $M$ is realizable as a hyperplane arrangement over the complex numbers, De Concini and Procesi's \emph{wonderful compactification} of the arrangement complement is a smooth complex projective variety, whose Chow ring $A^*(M)$ satisfies the \emph{K\"ahler package}. 
Feichtner and Yuzvinsky \cite{FY} gave an elegant combinatorial presentation for this Chow ring. There are natural classes $\alpha$ and $\beta$ in $A^*(M)$ whose intersection numbers deg$(\alpha^k \beta^{r-1-k})$ equal the $f$-vector above for $k=0, 1 \ldots, r-1$. The Hodge-Riemann relations then imply the desired log-concavity result.

When $M$ is not realizable, there seems to be no algebro-geometric context for this proof, but there is a tropical substitute: Ardila and Klivans's \emph{Bergman fan} $\Sigma_M$, which is a triangulation of the tropical linear space Trop$(M)$ of $M$. 
Its Chow ring $A^*(M)$ coincides with the Chow ring above in the realizable case. 
The approach above can be ``tropicalized" to include all matroids, but there are significant new hurdles to overcome. The main technical result of Adiprasito--Huh--Katz \cite{AHK} is that this combinatorial Chow ring $A^*(M)$ still satisfies the K\"ahler package for all $M$, even in the absence of algebraic geometry.  The main combinatorial result is that the intersection numbers deg$(\alpha^k \beta^{r-1-k})$ still equal the desired $f$-vector; this is an algebraic combinatorial computation in terms of the flags of flats of $M$, which correspond to the cones of the Bergman fan $\Sigma_M$.

\medskip

\noindent \textsf{\textbf{The conormal Chow ring, the conormal fan, and $h$-vectors.}} Huh \cite{Huh2} and Ardila--Denham--Huh \cite{ADH1} recently proved stronger conjectures from the 1980s by Brylawski and Dawson. The strongest is that the $h$-vector $h_0, h_1 \ldots, h_r$ of the broken circuit complex of a matroid $M$ is log-concave. 

When $M$ is representable over a field of characteristic zero, Huh identified two classes $\gamma$ and $\delta$ in the Chow ring of Varchenko's \emph{variety of critical points} and proved that the intersection numbers deg$(\gamma^k \delta^{n-2-k})$ now give the desired $h$-vector, which is log-concave by the Hodge-Riemann relations.

When $M$ is not realizable, the key tropical geometric object is the \emph{conormal fan} $\Sigma_{M, M^\perp}$ of the matroid. Ardila--Denham--Huh \cite{ADH1} showed that the resulting Chow ring $A^*(M,M^\perp)$ still satisfies the K\"ahler package. 
Combinatorially, the proof that the degrees deg$(\gamma^k \delta^{n-2-k})$ still give the desired $h$-vector is now much more intricate. It involves giving a Lagrangian interpretation of the Chern-Schwartz-MacPherson classes of the matroid, and studying the combinatorics of the \emph{biflags of biflats} of $M$, which correspond to the cones of the conormal fan $\Sigma_{M,M^\perp}$.

\subsection{The conormal fan: an origin story}\label{sec:whybiperm}

A central question in Ardila-Denham-Huh's program was the following: How should one define the conormal fan $\Sigma=\Sigma_{M,M^\perp}$ and the corresponding Chow ring $A^*(M,M^\perp)$ of a matroid $M$? This is the question that led to the harmonic polytope and the bipermutahedron, as we now explain.

When $\M$ is the matroid of a subspace $V$ of $\mathbb{C}^E$, the conormal fan $\Sigma_{\M, \M^\perp}$ is a tropical model of the projectivized conormal bundle of $V$. Since $\M^\perp$ is the matroid of the orthogonal complement of $V$, we expect the conormal fan to be supported on $\text{Trop}(\M) \times \text{Trop}(\M^\perp) \subset \N_n \times \N_n$, where $\N_n = \R^n / \R$. 
A desirable  fan structure $\Sigma$ on this support should have the following properties:

\noindent 1. There are classes $\gamma$ and $\delta$ in its Chow ring whose intersection numbers give the desired $h$-vector.

\noindent 2. The Chow ring is tractable for algebraic combinatorial computations, so we can prove 1.

\noindent 3. The fan is a subfan of the normal fan of a polytope, so its ample cone is nonempty.

\noindent 4. The fan is Lefschetz, so we can derive the desired log-concavity results.

Requirement 4. is resolved in \cite{ADH1} by showing that being Lefschetz only depends on the support of the fan -- and  $\text{Trop}(\M) \times \text{Trop}(\M^\perp)$ is the support of a  Lefschetz fan $\Sigma_{\M} \times \Sigma_{\M^\perp}$ by \cite{AHK} -- and not on the fan structure that we choose. Thus we can focus on the first three.

Requirement 2. is stated imprecisely, but a very desirable initial property is that our fan $\Sigma$ is simplicial. In this case the Chow ring $A(\Sigma)$ of the toric variety $X(\Sigma)$ has an algebraic combinatorial presentation due to Brion \cite{Brion}, and an interpretation in terms of piecewise polynomial functions due to Billera \cite{Billera}.
These results make it possible to carry out intersection-theoretic computations in this Chow ring.
Thus the first fan structure on $\text{Trop}(\M) \times \text{Trop}(\M^\perp)$ that we might try to use is the product of Bergman fans $\Sigma_{\M} \times \Sigma_{\M^\perp}$. It is simplicial, it does have a nice combinatorial structure, and it is a subfan of the normal fan of the product of permutohedra $\Pi_n \times \Pi_n$, addressing 2-4.

However, requirement 1. poses a problem. Relying on the geometry of the representable case, we expect the classes $\gamma$ and $\delta$ in the conormal Chow ring $A^*(\M, \M^\perp)= A^*(\Sigma)$ should be the pullbacks of a piecewise linear function $\alpha$ on $N_E$ under the maps
\[
\begin{array}{ccccc}
\pi: \Sigma \longrightarrow \Sigma_M &  \qquad & 
\sigma: \Sigma \longrightarrow \Gamma_n \\
\pi(x,y) = x & \qquad & 
\sigma(x,y) = x+y
\end{array}
\]
where $\Gamma_n$ is the reduced normal fan of the standard simplex and $\Sigma$ is our desired fan structure on $\text{Trop}(\M) \times \text{Trop}(\M^\perp)$. Here the piecewise linear function $\alpha$ can be regarded as a class in the Chow ring of the matroid $A^*(M)$ (where it is the class $\alpha$ of \cite{AHK}) or in the Chow ring of $\Gamma_n$.
 If $\Sigma$ equals $\Sigma_{\M} \times \Sigma_{\M^\perp}$ or any refinement of it, the first map is a map of fans, and $\gamma$ \emph{is} well defined. However, the second map is \emph{not} a map of fans for $\Sigma = \Sigma_{\M} \times \Sigma_{\M^\perp}$. Thus the product fan structure will not serve our purposes; we need to subdivide it further. How might we do this?

At this point, it is instructive to return to the case of tropical linear spaces above. 
In that case, one wants a similarly convenient fan structure for the tropical linear space $\text{Trop}(\M)$. Fortunately, one can do this for all matroids on $[n]$ at once, by intersecting $\text{Trop}(\M)$ with the permutohedral fan $\Sigma_n$. The result is the \emph{Bergman fan} $\Sigma_{\M}$ of $\M$, and the intersection theoretic computations in $\text{Trop}(M)$ become computations with flags of flats.

Similarly, we might try to find a suitable fan structure of $\text{Trop}(\M) \times \text{Trop}(\M^\perp)$ for all matroids $\M$ on $[n]$ simultaneously, by intersecting them with an appropriate complete fan. There is a minimal candidate: the coarsest common refinement of the product of permutohedral fans $\Sigma_n \times \Sigma_n$ -- which induces the fan structure $\Sigma_{\M} \times \Sigma_{\M^\perp}$ --  
and $\sigma^{-1}(\Gamma_n)$ --  the coarsest fan that guarantees that the class $\delta$ is well-defined. The resulting fan is the \emph{harmonic fan}.

The harmonic fan  is the reduced normal fan of a polytope, namely, the Minkowski sum
\[
H_{n,n} := (\Pi_n \times \Pi_n) + D_n,
\]
of the product of two permutohedra $\Pi_n \times \Pi_n$ and the diagonal simplex $D_n = \conv\{\e_i + \f_i\}_{i \in E}$. Thus requirement 3. above is satisfied. The resulting polytope is the \emph{harmonic polytope}.

\medskip

\noindent \textsf{\textbf{Combinatorial models for the Lagrangian geometry of matroids}}
The harmonic polytope has a drawback for our geometric purposes: it is not simple, so the resulting fan structure on $\text{Trop}(\M) \times \text{Trop}(\M^\perp)$ is not simplicial, posing numerous obstacles. 
Thus we wish to find a simple polytope that has the harmonic polytope $H_{n,n}$ as a Minkowski summand, and has simple enough combinatorial structure that we can carry out computations. Ardila--Denham--Huh propose the \emph{bipermutohedron} $\Pi_{n,n}$ as a solution; the combinatorics of this polytope is studied in detail in \cite{biperm}. Its faces are indexed by \emph{biflags} of subsets of $[n]$. 

The \emph{conormal fan} $\Sigma_{M,M^\perp}$ is then defined as the intersection of $\text{Trop}(\M) \times \text{Trop}(\M^\perp)$ with the bipermutahedral fan $\Sigma_{n,n}$; its faces are indexed by the \emph{biflags of biflats} of $M$. The resulting intersection-theoretic computations in the Chow ring $A^*(M,M^\perp)$ require an intricate, interesting analysis of these biflags; this can be done using a tropical geometric approach \cite[Sections 3,4]{ADH1} involving Chern-Schartz-MacPherson classes. There is also an algebraic combinatorial approach \cite{ADH2}  involving an intricate analysis of the biflags of biflats of ordered matroids. Both of these approaches require significant new ideas and lead to new developments.

As far as we know, there is nothing canonical about the choice of the bipermutahedron $\Pi_{n,n}$ above. It is natural to wonder whether there are alternative, perhaps easier approaches: \textbf{Are there other simple polytopes in whose normal fans we could carry out these Lagrangian geometric computations on matroids?} 
On the other hand, the harmonic polytope $H_{n,n}$ is canonical: \textbf{Every such simple polytope in $\N_n \times \N_n$ must contain the harmonic polytope as a summand.}

For every such simple polytope $P$ in $\N_n \times \N_n$ that we can find, we expect that the program above will produce a combinatorial model for the Lagrangian geometry of matroids on $[n]$. The building blocks of this model will be given by the face structure of the polytope $P$, and how it interacts with each matroid $M$ on $[n$]. The resulting intersection theoretic computations will teach us about the Lagrangian combinatorics of matroids. This seems to be a direction of study worth pursuing further.

\section{The harmonic polytope} \label{sec;polytope}

Having motivated the study of the harmonic polytope, we now analyze it in detail. 
Let $n$ be a positive integer and let $[n]:=\{1,\ldots, n\}$.
Consider two copies of $\R^n$ with respective standard bases  $\{\e_i \, : \, i \in [n]\}$ and $\{\f_i \, : \, i \in [n]\}$. For any subset $S$ of $[n]$, we write 
\[
\mathbf{e}_S=\sum_{i \in S} \mathbf{e}_i, \qquad \mathbf{f}_S=\sum_{i \in S} \mathbf{f}_i.
\]
We also consider the $(n-1)$-dimensional vector space $\N_n := \mathbb{R}^n/\R\e_{[n]}$.

The (inner) normal fan $\NN(P)$ of a polytope $P \subset \R^n$ is a complete fan in the dual space $(\R^n)^*$ whose cones are
\[
\NN(P)_Q := \{w \in (\R^n)^* \, : \, P_w \supseteq Q\}
\]
for each nonempty face  $Q$ of $P$, where $P_w = \{x \in P \, : \, w(x) = \min_{y \in P} w(y)\}$ is the $w$-minimal face of $P$. 
The face poset of the 
normal fan of $P$ is isomorphic to the reverse of the face poset of $P$.
The relative interior of a cone $\sigma$ is the interior of $\sigma$ inside its affine span. In particular, the relative interior of $\NN(P)_Q$ is 
\[
\NN(P)_Q^\circ := \{w \in (\R^n)^* \, : \, P_w = Q\}.
\]
The chambers of $\NN(P)$ are the cones of maximal dimension.

The normal fan of the permutohedron 
\[
\Pi_n = \conv\Big\{(x_1,\ldots,x_n)\ | \ \text{$x_1,\ldots,x_n$ is a permutation of $[n]$} \Big\} \subseteq \mathbb{R}^n
\]
is the \emph{permutohedral fan} $\Sigma_n\subset \N_n$,  also known as the \emph{braid fan} or the \emph{type $A$ Coxeter complex}. It is the complete simplicial fan in $\R^n$ whose chambers are cut out by the $n$-dimensional \emph{braid arrangement}, the real hyperplane arrangement in $\R^n$ consisting of the ${n \choose 2}$ hyperplanes 
\[
z_i = z_j,  \ \ \text{for distinct elements $i$ and $j$ of $[n]$.}
\]
The face of the permutohedral fan containing a given point $z$ in its relative interior is determined by the relative order of its homogeneous coordinates $(z_1,\ldots,z_n)$. 

Let $D_n$ be the $(n-1)$-dimensional simplex,
\[
D_n := \, \conv \Big\{\e_i + \f_i \, : \, i \in [n] \Big \} \subseteq \R^n \times \R^n. \
\]
The normal fan of the simplex $D_{n}$ is the simplicial fan $\Delta_{n,n}$ whose $n$  chambers are the cones
\[
\mathscr{C}_k = \Big\{(z, w) \in \R^n \times \R^n  \, | \, \min_{i \in [n]} (z_i+w_i) = z_k+w_k\Big\}. 
\]

Recall that the \emph{Minkowski sum} and \emph{Minkowski difference} of polytopes $P$ and $Q$ in $\R^d$ are
\[
P+Q = \{p+q \, : \, p \in P, \, q \in Q\}, \qquad
P-Q = \{r \in \R^d \, : \, r+Q \subseteq P\}.
\]
The following polytope is our main object of study. 
 
\begin{definition}\label{def:polytope}
The \emph{harmonic polytope} is the Minkowski sum
\[
H_{n,n} := D_n + (\Pi_n \times \Pi_n) \subset \R^n \times \R^n.
\]
The \emph{harmonic fan} is its reduced normal fan $\mathcal{N}(H_{n,n})$ in $\N_n \times \N_n$. 
\end{definition}

Figure \ref{fig:hexagon} shows the harmonic polytope $H_{2,2}$ and its reduced normal fan. 
The normal fan of a Minkowski sum of two polytopes is the coarsest common refinement of their normal fans, see e.g. \cite[Proposition 7.12]{Ziegler}.
Therefore, the normal fan of $H_{n,n}$ is the coarsest common refinement of the normal fans of $D_n$ and $\Pi_n \times \Pi_n$. 
Its lineality space is $\R\{\e_{[n]}, \f_{[n]}\}$.

\subsection{The face structure of the harmonic polytope.}

The cone of the harmonic fan containing a point $(z,w) \in \N_n \times \N_n$ is determined by:

$\bullet$ the set of indices $i$ for which the minimum of $z_i + w_i$ is attained,

$\bullet$ the \textbf{reverse}\footnote{Of course, this is the same information as the relative order of the $z_i$s. We use the reverse order because it is consistent with our choice of working with inner normal fans.} relative order of the $z_i$s, and 

$\bullet$ the reverse relative order of the $w_i$s.

\noindent Our next task is to characterize the triples that arise in this way.

Recall that an \emph{ordered set partition} of $[n]$ is a sequence $\pi=E_1|\cdots|E_\ell$ such that $E_1 \cup \cdots \cup E_\ell = [n]$ and $E_i \cap E_j = \emptyset$ for all $i \neq j$. The \emph{length} of $\pi$ is $\ell(\pi) := \ell$.
The ordered set partitions of $[n]$ form a poset under \emph{adjacent refinement}, where $\pi \leq \pi'$ if every block of $\pi'$ is a union of a set of consecutive blocks of $\pi$. For example $14|3|26|8|57 \leq 134 | 26 | 578$.

\begin{definition} \label{def:harmonictriples}
The poset of harmonic triples $HT_n$ is defined as follows:
\begin{enumerate}
\item
A \emph{harmonic triple} $\tau = (K; \pi_1, \pi_2)$ on $[n]$ consists of a nonempty subset $K \subseteq [n]$ and two ordered set partitions $\pi_1$ and $\pi_2$ of $[n]$ such that
\begin{enumerate}
\item The restrictions $\pi_1|K$ and $\pi_2|K$ of $\pi_1$ and $\pi_2$ to $K$ are opposite to each other, and
\item If $j \notin K$ appears in the same or a later block than $k \in K$ in one of the set partitions $\pi_1$ and $\pi_2$, then $j$ must appear in an earlier block than $k$ in the other set partition.
\end{enumerate}

\item
The \emph{poset of harmonic triples} $HT_n$ is defined by setting $(K; \pi_1, \pi_2) \leq (K'; \pi'_1, \pi'_2)$ if and only if $K \subseteq K'$, $\pi_1$ is an adjacent refinement of $\pi'_1$, and $\pi_2$ is an adjacent refinement of $\pi'_2$. 

\item
A \emph{fine harmonic triple} is a minimal element of the poset $HT_n$. A \emph{coarse harmonic triple} is a maximal element of $HT_n - \{\widehat{1}\}$.
\end{enumerate}
\end{definition}

Notice that the maximum element $\widehat{1}$ of $HT_n$ is the triple $([n], [n], [n])$. The fine harmonic triples are the minimal elements, for which $K$ consists of a single element $k$, and $\pi_1$ and $\pi_2$ only have blocks of size $1$ -- and hence may be thought of as permutations in one-line notation.

\begin{example}\label{ex:triples}
Consider the triple $(\textbf{3467}, \mathbf{4}5|8|2|1\mathbf{37}9|\mathbf{6}, \mathbf{6}|1|59|2\mathbf{3}\mathbf{7}|8|\mathbf{4})$, were we omit the brackets and write the elements of $K$ in bold for easier readability. The reader is invited to verify that this triple satisfies the required conditions to be harmonic.
On the other hand, $j=1$ and $k=3$ do not satisfy condition (b) in the non-harmonic triple $(\textbf{3467}, \mathbf{4}5|8|2|1\mathbf{37}9|\mathbf{6}, \mathbf{6}|5|2\mathbf{3}\mathbf{7}|89|1\mathbf{4})$.
\end{example}

\begin{proposition} \label{prop:faces}
The combinatorial structure of the harmonic fan $\mathcal{N}(H_{n,n})$  is as follows.
\begin{enumerate}
\item
The cones of the harmonic fan are in bijection with the harmonic triples on $[n]$.

\item 
The dimension of the cone labeled by $\tau = (K; \pi_1, \pi_2)$ is $\ell(\pi_1) + \ell(\pi_2) - \ell(\pi_1|K) - 1$.

\item
Two cones $\sigma$ and $\sigma'$ of the harmonic fan satisfy $\sigma \supseteq \sigma'$ if and only if their harmonic triples satisfy $\tau \leq \tau'$ in $HT_n$.
\end{enumerate}
\end{proposition}

\begin{proof}
1. Given a cone $\sigma$ of the harmonic fan, we define the triple $\tau(\sigma)$ as follows. Let $(z,w)$ be an interior point of $\sigma$. We let $K$ be the set of indices $k$ for which the minimum of $z_k + w_k$ is attained, $\pi_1$ be the partition encoding the reverse relative order of the $z_i$s, and $\pi_2$ be the reverse relative order of the $w_i$s. For example, we have, for the following cone $\sigma$, 
\begin{eqnarray*}
z_k+w_k \text{ is minimum for } k=3,4,6,7 & & \\
z_6 < z_1=z_3=z_7=z_9 < z_2 < z_8 < z_4=z_5
& \mapsto& \tau(\sigma) = (\textbf{3467}, \mathbf{4}5|8|2|1\mathbf{37}9|\mathbf{6}, \mathbf{6}|1|59|2\mathbf{3}\mathbf{7}|8|\mathbf{4}). \\ 
w_4 < w_8 < w_2=w_3=w_7<w_5=w_9<w_1<w_6 & &
\end{eqnarray*}

Since $z_k+w_k$ is constant for $k$ in $K$, the relative order of the $z_k$s is exactly the opposite of the relative order of the $w_k$s, so (a) holds. Also, if $j \notin K$ appears in the same or a later block than $k \in K$ in, say the first set partition, then we have $z_k \geq z_j$. But then $z_k+w_k < z_j + w_j$ implies that $w_k < w_j$, so $j$ must appear before $k$ in the second set partition. Therefore (b) also holds.

\smallskip

Conversely, suppose $\tau=(K;\pi_1, \pi_2)$ is a harmonic triple, and let us construct a point $(z,w)$ whose associated triple is $\tau$. We begin by defining the values of $z_k$ and $w_k$ for $k \in K$. We let $z_k=a$ where $k$ is in the $a$th block of $\pi_2|K$ and $w_k=b$ where $k$ is in the $b$th block of $\pi_1|K$. Then the $z_k$s and $w_k$s are in the order specified by $\pi_1|K$ and $\pi_2|K$, respectively, and, since $\pi_1|K$ and $\pi_2|K$ are opposites of each other, $z_k+w_k=c$ where $\pi_1|K$ and $\pi_2|K$ have $c-1$ blocks. 

Now define the values of $z_j$ for $j \notin K$ as follows. If $j$ is in the same block of $\pi_1$ as $k \in K$ set $z_j=z_k$.
Define the remaining entries $z_j$ to have the order stipulated by $\pi_1$, while making each one of them very large -- say, within a small $\epsilon>0$ of the first entry $z_k$ such that $z_k > z_j$, if there is one. For example, for the triple $\tau= (\textbf{3467}, \mathbf{4}5|8|2|1\mathbf{37}9|\mathbf{6}, \mathbf{6}|1|59|2\mathbf{3}\mathbf{7}|8|\mathbf{4})$ of Example \ref{ex:triples}, we may set
\[
z_{\mathbf 6} = \mathbf{1} < z_1=z_{\mathbf 3}=z_{\mathbf 7}=z_9=\mathbf{2} < z_2=2.8 < z_8 = 2.9 < z_{\mathbf 4}=z_5=\mathbf{3}
\]
\[
w_{\mathbf 4} = \mathbf{1} < w_8=1.9 < w_2 = w_{\mathbf 3}= w_{\mathbf 7}=\mathbf{2} <  w_5 = w_9 = 2.8< w_1  = 2.9 < w_{\mathbf 6} = \mathbf{3}.
\]
By construction, the order of the $z_i$s (resp. the $w_i$s) is the opposite of the order dictated by $\pi_1$ (resp. $\pi_2$). Also $z_k+w_k=c$ is constant for $k \in K$. It remains to show that $z_j+w_j >c$ for $j \notin K$. Assume contrariwise that $z_j+w_j \leq c$. Then for any $k \in K$ we must have $z_j \leq z_k$ or $w_j \leq w_k$. Assume it is the former, and choose $k \in K$ where $z_k$ is minimum such that $z_k \geq z_j$. By construction, we have $z_j > z_k-\epsilon$. Furthermore $j$ comes after $k$ in $\pi_1$, so it must come before $k$ in $\pi_2$; by construction, we have $w_j > (w_k+1)-\epsilon$. Thus $z_j+w_j > c+1-2\epsilon > c$, a contradiction.
We conclude that $\tau$ is the label of a cone of the harmonic fan containing $(z,w)$, as desired.

\medskip

2. The set of 
points $(z,w)\in\N_n \times \N_n$ that give rise to the ordered set partitions $\pi_1$ and $\pi_2$ have $(\ell(\pi_1)-1) + (\ell(\pi_2)-1)$ degrees of freedom. The condition that $z_k+w_k$ are equal for all $k \in K$ introduces $\ell(\pi_1|K)-1 = \ell(\pi_2|K)-1$ linear constraints.

\medskip

3. To go up the face poset from the cone indexed by $(K; \pi_1, \pi_2)$, we need to turn some of the defining equalities into inequalities. The effect of this on the label is to remove elements from $K$ and break a parts of $\pi_1$ and $\pi_2$  into adjacent parts.
\end{proof}

Using \Cref{prop:faces}, one may check that the harmonic fan is neither simple nor simplicial, already for $n=3$. We now give the vertex and inequality description of the harmonic polytope.

\begin{proposition} \label{prop:vertices}
The number of vertices of the harmonic polytope $H_{n,n}$ is
\[
v(H_{n,n}) = (n!)^2 \left(1 + \frac12 + \frac13 + \cdots + \frac1n \right).
\]
\end{proposition}

\begin{proof}
By Proposition \ref{prop:faces} we need to count the fine harmonic triples $\tau = (K; \pi_1, \pi_2)$; these are the ones where $K = \{k\}$ and both $\pi_1$ and $\pi_2$ are permutations.
To specify $\tau$, we first specify the element $k$. Out of the remaining $n-1$ elements, we choose 
which $a$ of them precede $k$ in  $\pi_1$ and follow $k$ in $\pi_2$,
which $b$ of them precede $k$ in  both $\pi_1$ and $\pi_2$, and
which $c$ of them follow $k$ in  $\pi_1$ and precede $k$ in $\pi_2$. 
Finally we choose the order of the $a+b$ elements preceding $k$ in $\pi_1$, the order of the $c$ elements following $k$ in $\pi_1$, the order of the $b+c$ elements preceding $k$ in $\pi_2$, and the order of the $a$ elements following $k$ in $\pi_2$. It follows that
\begin{eqnarray*}
v(H_{n,n}) &=& n \sum_{a+b+c = n-1} {n-1 \choose a,b,c} (a+b)! \, c! \, a! \, (b+c)! \\
&=& n! \sum_{a+b+c=n-1} \frac{(a+b)!(b+c)!}{b!} \\ 
&=& n!  \,\sum_{a=0}^{n-1} \left( a!(n-1-a)! \sum_{b=0}^{n-1-a} {a+b \choose a} \right)\\
&=& n! \, \sum_{a=0}^{n-1} \left(a!(n-1-a)! {n \choose a+1} \right)\\
&=& (n!)^2 \, \sum_{a=0}^{n-1} \frac1{a+1},
\end{eqnarray*}
as desired.
\end{proof}

Let us give a concrete description of the vertices of $H_{n,n}$.

\begin{proposition}\label{prop:vertexcoords}
The vertices of the harmonic polytope $H_{n,n}$ are 
\[
v_\tau = \e_k + \f_k + (\pi_1^{-1},0) + (0, \pi_2^{-1})
\]
for the fine harmonic triples $\tau = (k; \pi_1, \pi_2)$ on $[n]$, where $\pi^{-1}$ denotes the inverse of the permutation $\pi$ in one-line notation.
\end{proposition}

\begin{proof}
Consider a point $(z,w)$ in the interior of the chamber of the normal fan $\NN(H_{n,n})$ corresponding to a fine harmonic triple $\tau = (k; \pi_1, \pi_2)$. The minimal vertex of $H_{n,n}$ in the direction $(z,w)$ is
\begin{eqnarray*}
(H_{n,n})_{z,w} &=& (D_n)_{(z,w)} + (\Pi_n \times 0)_{(z,w)} + (0 \times \Pi_n)_{(z,w)} \\
&=& (\e_k + \f_k) + (\pi_1^{-1},0) + (0, \pi_2^{-1})
\end{eqnarray*}
as desired.
\end{proof}

For example, Figure \ref{fig:hexagon} shows how the harmonic polytope $H_{2,2}$ sits in the lattice $\Z^2 \times \Z^2$. Its inner normal fan is the harmonic fan, which coincides with the bipermutohedral fan (only) for $n=2$; the orientation shown here matches the one in \cite[Figure 4]{ADH1}.

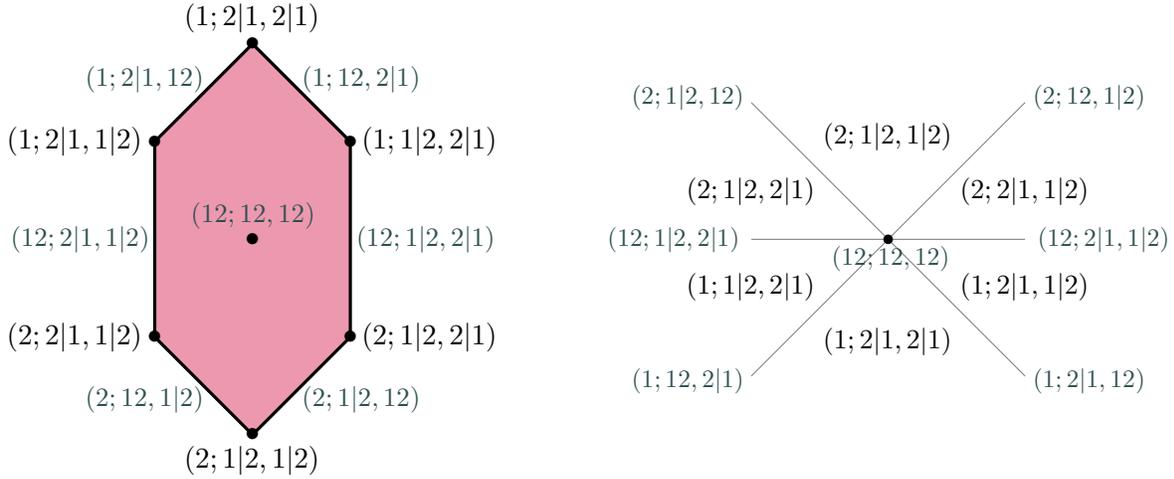
	\begin{figure}[h] 
	\centering
	\begin{tikzpicture}[scale=1.3]
root/.style={circle,draw,inner sep=1.2pt,fill=black}]
	\filldraw[fill=ganador!50,very thick] 
(0,0) node[very thick] {$\bullet$} 
node [below] {$( 2; 1| 2,1| 2)$}--
node [below right=-4] {\gris{\small{$(2; 1|2, 12)$}}} 
(1,1) node {$\bullet$} node [right] {$( 2; 1|2,2|1)$}--
node [right=-2] {\gris{\small{$( 12; 1|2, 2| 1)$}}} 
(1,3) node {$\bullet$} node [right] {$( 1; 1|2,2|1)$}--
node [above right=-4] {\gris{\small{$( 1; 12, 2|1)$}}}
(0,4) node {$\bullet$} node [above] {$(1; 2| 1,2| 1)$} --
node [above left=-4] {\gris{\small{$( 1; 2|1, 12)$}}} 
(-1,3) node {$\bullet$} node [left] {$( 1; 2|1,1|2)$}--
node [left=-2] {\gris{\small{$( 12; 2|1, 1|2)$}}} 
(-1,1) node {$\bullet$} node [left] {$( 2; 2|1, 1|2)$}--
node [below left=-4] {\gris{\small{$( 2; 12, 1|2)$}}} 
(0,0);
\draw (0,2) node {$\bullet$} node [above] {\gris{{\small{$(12; 12, 12)$}}}};
 \begin{scope}[shift={(6.5,2)}, scale=.7, root/.style={circle,draw,inner sep=1.2pt,fill=black},every node/.style={scale=0.95}]
\draw[style=thin,color=gray] (2,0) -- (-2,0);
\draw[style=thin,color=gray] (-2,-2) -- (2,2);
\draw[style=thin,color=gray] (2, -2) -- (-2, 2);
\node at (0,-1.5) {$(1; 2|1, 2|1)$};
\node at (0,1.5) {$(2; 1|2, 1|2)$};
\node at (2,-0.7) {$(1; 2|1, 1|2)$};
\node at (-2,-0.7) {$(1; 1|2, 2|1)$};
\node at (2,0.7) {$(2; 2|1, 1|2)$};
\node at (-2,0.7) {$(2; 1|2, 2|1)$};
\node at (0,0) [root, label={[label distance=-3pt]below:{\gris{\small{$\,(12; 12, 12)$}}}}] {};
\node at (1.8,-1.45) [label={[label distance=5pt]300:{\gris{\small{$(1; 2|1, 12)$}}}}] {};
\node at (1.8,1.45) [label={[label distance=5pt]60:{\gris{\small{$(2; 12, 1|2)$}}}}] {};
\node at (-1.8,-1.45) [label={[label distance=5pt]240:{\gris{\small{$(1; 12, 2|1)$}}}}] {};
\node at (-1.8,1.45) [label={[label distance=5pt]120:{\gris{\small{$(2; 1|2, 12)$}}}}] {};
\node at (1.9,0) [label=right:{\gris{\small{$(12; 2|1, 1|2)$}}}] {};
\node at (-1.9,0) [label=left:{\gris{\small{$(12; 1|2, 2|1)$}}}] {};
\node (1) at (3,0) {};
\node (2) at (4.5,0) {};
 \end{scope}
\end{tikzpicture}
\vspace{.5cm}

\begin{tabular}{|c||c|c|c|c|c|c|}\hline
$(k;\pi_1,\pi_2)$ & 
$( 2;1| 2,1| 2)$ &
$( 2;1| 2,2| 1)$ &
$( 2;2| 1,1| 2)$ &
$( 1;1| 2,2| 1)$ &
$( 1;2| 1,1| 2)$ &
$( 1;2| 1,2| 1)$ \\
\hline
\hline
$\begin{pmatrix} x_1 & x_2 \\ y_1 & y_2  \end{pmatrix}$  &
$\begin{pmatrix} 1 & 3 \\ 1 & 3 \end{pmatrix}$  &
$\begin{pmatrix} 1 & 3 \\ 2 & 2  \end{pmatrix}$  &
$\begin{pmatrix} 2 & 2 \\ 1 & 3  \end{pmatrix}$  &
$\begin{pmatrix} 2 & 2 \\ 3 & 1  \end{pmatrix}$  &
$\begin{pmatrix} 3 & 1 \\ 2 & 2 \end{pmatrix}$  &
$\begin{pmatrix} 3 & 1 \\ 3 & 1 \end{pmatrix}$  \\
\hline
\end{tabular}

\caption{The harmonic polytope $H_{2,2}$ in $\Z^2 \times \Z^2$ and its reduced normal fan. The faces correspond to the harmonic triples on $[2]$. The table lists the fine harmonic triples $(k; \pi_1, \pi_2)$ and the corresponding vertices of $H_{2,2}$. \label{fig:hexagon}}
\end{figure}

For a larger example, the vertex of the harmonic polytope $H_{5,5}$ corresponding to the fine harmonic triple $\tau= (4; 53412, 14352)$ on $[5]$ is
\[
v_{T} = 
\begin{pmatrix}
0 & 0 & 0 & 1 & 0 \\
0 & 0 & 0 & 1 & 0 
\end{pmatrix}
+
\begin{pmatrix}
4 & 5 & 2 & 3 & 1 \\
0 & 0 & 0 & 0 & 0 
\end{pmatrix}
+
\begin{pmatrix}
0 & 0 & 0 & 0 & 0 \\
1 & 5 & 3 & 2 & 4 
\end{pmatrix}
= 
\begin{pmatrix}
4 & 5 & 2 & 4 & 1 \\
1 & 5 & 3 & 3 & 4 
\end{pmatrix}
\]
since $53412^{-1} = 45231$ and $14352^{-1} = 15324$.

\begin{proposition} \label{prop:facets}
The number of facets of the harmonic polytope $H_{n,n}$ is $3^n-3$.
\end{proposition}

\begin{proof}
In light of Proposition \ref{prop:faces} we need to enumerate the coarse harmonic triples; that is, those for which $\ell(\pi_1) + \ell(\pi_2) - \ell(\pi_1|K) - 1 = 1$. We consider three cases.

\smallskip

(i) $\ell(\pi_1)=1$: In this case $\ell(\pi_1|K)=1$ so we must have $\ell(\pi_2)=2$; say $\pi_2 = S|T$. Then we have $\tau=(K;[n],S|T)$, and for this triple to be harmonic we must have $K=T$. Therefore $\tau = (T; [n],S|T)$. The corresponding ray of the harmonic fan is $\e_{[n]} + \f_T$.

\smallskip

(ii) $\ell(\pi_2)=1$: Similarly we obtain $\tau = (T; S|T, [n])$. The corresponding ray of the harmonic fan is $\e_S + \f_{[n]}$.

\smallskip

(iii) $\ell(\pi_1) >1$ and $\ell(\pi_2) > 1$: Since $(\ell(\pi_1) - \ell(\pi_1|K)) + (\ell(\pi_2) - 2) = 0$ and both summands are nonnegative, we must have $\ell(\pi_2)=2$ and $\ell(\pi_1) = \ell(\pi_1|K)$. Similarly $\ell(\pi_1)=2$ and $\ell(\pi_2) = \ell(\pi_2|K)$. Let us write
\[
\pi_1 = S|S', \qquad \pi_2 = T|T', \qquad \text{and} \qquad \pi_1|K = K_S|K_{S'}, \qquad \pi_2|K = K_T|K_{T'}.
\]
Since $K_S|K_{S'} = K_{T'}|K_T$, an element $j \in S' - K_{S'}$ would contradict Definition \ref{def:harmonictriples}.1(b), so we must have $S' = K_{S'} = K_T$. Similarly $T' = K_{T'} = K_S$. Then $K = K_S \cup K_{S'} = S' \cup T'$. Also $S' \cap T' = K_S \cap K_{S'} = \emptyset$, so $S \cup T = [n]$. Thus
\[
\tau = ([n]-(S \cap T); \, S|([n]-S), \, T|([n]-T)) \qquad \text{ for } S \cup T = [n]. 
\]
The corresponding ray of the normal fan is $\e_S + \f_T$.

We conclude that the rays of the harmonic fan are the vectors $\e_S + \f_T$ where $S$ and $T$ are non-empty, they are not both equal to $[n]$, and $S \cup T = [n]$. There are $3^n-3$ such vectors because we can choose freely, for each $i \in [n]$, whether (a) $i$ is in $S$ and not $T$, (b) $i$ is in $T$ and not $S$, or (c) $i$ is in both $S$ and $T$; but the three pairs $([n], \emptyset), (\emptyset, [n])$ and $([n],[n])$ are invalid.
\end{proof}

As introduced in \cite{ADH1}, a \emph{bisubset} of $[n]$ is a pair $S|T$ of nonempty subsets of $[n]$, not both equal to $[n]$, such that $S \cup T = [n]$. The previous proof shows that they are in correspondence with the facets of $H_{n,n}$. More precisely, we have:

\begin{proposition}\label{prop:inequalities}
The harmonic polytope $H_{n,n}$ is given by the following minimal inequality description:
\begin{eqnarray*}
\sum_{e \in [n]} x_e &=& \frac{n(n+1)}2 + 1, \\
\sum_{e \in [n]} y_e &=& \frac{n(n+1)}2 + 1, \\
\sum_{s \in S} x_s + \sum_{t \in T} y_t &\geq& \frac{|S|(|S|+1) + |T|(|T|+1)}2  + 1 \qquad \text{for each bisubset $S|T$ of $[n]$}.
\end{eqnarray*}
\end{proposition}

\begin{proof}
The first two equations hold, and determine a codimension two subspace perpendicular to the lineality space $\R\{\e_{[n]}, \f_{[n]}\}$ of $\NN(H_{n,n})$. The minimal inequality description is then determined by the rays $\e_S + \f_T$ for the bisubsets $S|T$. We have 
\begin{eqnarray*}
\min_{(x,y) \in H_{n,n}}(\e_S+\f_T)(x,y) &=& 
\min_{(x,y) \in D_{n}}(\e_S+\f_T)(x,y) + 
\min_{(x,y) \in \Pi_{n} \times 0}\e_S(x,0) + 
\min_{(x,y) \in 0 \times \Pi_{n}}\f_T(0,y)  \\
&=& 1 + (1+2+\cdots + |S|) + (1+2+\cdots+|T|),
\end{eqnarray*}
which implies the given description.
\end{proof}

We now offer an alternative description of the faces of the harmonic polytope, which gives rise to a formula for the $f$-vector.
A \emph{harmonic table} $\T$ on $[n]$ of size $\ell =: \ell(\T)$ is a triangular table having $2\ell+1$ rows and $2\ell+1$ columns of lengths $2\ell+1, 2\ell-1, 2\ell-1, \ldots, 3, 3, 1, 1$, respectively, decorated with the following data:

\noindent $\bullet$
A labeling of the even columns with nonempty, pairwise disjoint subsets of $[n]$. 

\noindent $\bullet$
A labeling of the even rows with the same subsets, listed in the opposite order.

\noindent $\bullet$
A placement of each element of $[n]$ not used as a row or column label in one box of the table.

We let $c_i(\T)$ and $r_i(\T)$ denote the number of elements in column $2i+1$ and row $2i+1$ of $\T$, respectively. Figure \ref{fig:harmonictable} shows a harmonic table of size $3$ on $[9]$, with $c_1(\T) = 2$, $r_1(\T)=3$, $r_2(\T) = 1$, and all other $c_i(\T)$ and $r_i(\T)$ equal to $0$.

\begin{figure}[h]
\centering
	\begin{tikzpicture}
	\draw (0,0)--(.5,0)--(.5,1)--(1.5,1)--(1.5,2)--(2.5,2)--(2.5,3)--(3.5,3)--(3.5,3.5)--(0,3.5)--(0,0)
		(.5,1)--(.5,3.5)
		(1,1)--(1,3.5)
		(1.5,2)--(1.5,3.5)
		(2,2)--(2,3.5)
		(2.5,3)--(2.5,3.5)
		(3,3)--(3,3.5)
		(0,.5)--(.5,.5)
		(0,1)--(.5,1)
		(0,1.5)--(1.5,1.5)
		(0,2)--(1.5,2)
		(0,2.5)--(2.5,2.5)
		(0,3)--(2.5,3)
		(1.25,1.25) node {8}
		(1.25,1.75) node {2}
		(.75,2.25) node {5}
		(1.75,2.25) node {19}
		(.75,3.75) node {\textbf{4}}
		(1.75,3.75) node {\textbf{37}}
		(2.75,3.75) node {\textbf{6}}
		(-.25,.75) node {\textbf{4}}
		(-.35,1.75) node {\textbf{37}}
		(-.25,2.75) node {\textbf{6}}
	;
\end{tikzpicture}
\caption{\label{fig:harmonictable} A harmonic table of size $3$ on $[9]$.}
\end{figure}

Let $F(m)$ be the $m$th \emph{Fubini number} (or \emph{ordered Bell number}), which counts the ordered set partitions of $[m]$. Also recall that the \emph{Stirling number of the second kind} $S(m,p)$ counts the number of unordered set partitions of $m$ into $p$ parts.

\begin{proposition}\label{prop:fvector}
The number of faces of the harmonic polytope $H_{n,n}$ is
\[
f(H_{n,n}) = \sum_{\T} \, \prod_{i=0}^{\ell(\T)} \Big(F(c_i(\T)) F(r_i(\T)) \Big)
\]
summing over all harmonic tables on $[n]$. The number of $d$-dimensional faces of $H_{n,n}$ is
\[
f_d(H_{n,n}) = \sum_{\T} \, \sum_{\mathsf{a}, \mathsf{b}} \,  \prod_{i=0}^{\ell(\T)} \Big(S(c_i(\T),a_i) \, a_i! \, S(r_i(\T),b_i) \, b_i! \Big)
\]
summing over all harmonic tables $\T$ on $[n]$ and all sequences $\mathsf{a}=(a_0, a_1, \ldots, a_\ell)$ and $\mathsf{b} = (b_0, b_1, \ldots, b_\ell)$ with $\ell=\ell(\T)$ and $\sum_i a_i + \sum_i b_i + \ell= 2n-d-1$.
\end{proposition}

\begin{proof}
A harmonic triple $T=(K; \pi_1, \pi_2)$ can be constructed in five steps, with the help of a harmonic table, as follows.
This process is illustrated in Figure \ref{fig:harmonictable}, which shows the harmonic table that gives rise to the harmonic triple $T=(\textbf{3467}, \mathbf{4}5|8|2|1\mathbf{37}9|\mathbf{6}, \mathbf{6}|1|59|2\mathbf{3}\mathbf{7}|8|\mathbf{4})$ of Example \ref{ex:triples}.

\smallskip

1. Choose the subset $K$ of $[n]$.

\smallskip

2. Choose the ordered set partition $\pi_1|K =: K_1 |\cdots | K_\ell$. This automatically determines $\pi_2|K$, which is its reverse. Record this data on a triangular table $\T$ of size $\ell$, labeling the $2i$th column from left to right and the $2i$th row from bottom to top with the set $K_i$.

\smallskip

3. Choose, for each element $j \in J := [n]-K$, its position relative to $K_1, \ldots, K_\ell$ in the ordered set partition $\pi_1$,  and its position relative to $K_1, \ldots, K_\ell$ in the ordered set partition $\pi_2$. Record this information in the table $\T$ as follows. 
If $j$ is in the same block as $K_i$ in $\pi_1$ (resp. in $\pi_2$), put it in the column (resp. row)  labeled by $K_i$ in $\T$.
If $j$ is between blocks $K_i$ and $K_{i+1}$ in $\pi_1$ (resp. in $\pi_2$), put it in the unlabeled column (rep. row) between the columns (resp. rows) labeled by $K_i$ and $K_{i+1}$ in $\T$.
Notice that, by item 1(b) in the Definition \ref{def:harmonictriples} of a harmonic triple, all these numbers will land inside the triangular table.

\smallskip

4. Choose the relative order of the elements of $J = [n]-K$ in the ordered set partition $\pi_1$. To do this, it suffices to choose, for each $i$, the relative order of the elements of $J$ that appear between blocks $K_i$ and $K_{i+1}$ of $\pi_1$ for each $i$. These are precisely the $c_i(\T)$ elements in column $2i+1$, and their order is given by an arbitrary ordered set partition of that size, so there are $F(c_1(\T)) \cdots F(c_\ell(\T))$ such choices.

\smallskip

5. Choose the relative order of the elements of $J = [n]-K$ in the ordered set partition $\pi_2$. As in step 4, there are $F(r_1(\T)) \cdots F(r_\ell(\T))$ such choices.

\smallskip

Each harmonic triple on $[n]$ -- and hence each face of the harmonic polytope $H_{n,n}$ -- arises in a unique way from this procedure. This proves the first formula.

\medskip

By Proposition \ref{prop:faces}.2, the $d$-dimensional faces of the harmonic polytope $H_{n,n}$ dimension $d$ correspond to the harmonic triples $(\tau; \pi_1, \pi_2)$ with $d = \ell(\pi) + \ell(\pi_2) - \ell(\pi_1|K) - 1$.
For the harmonic table $\T$ given by $(\tau; \pi_1, \pi_2)$, let $a_i$ (resp. $b_i$) denote the length of the ordered set partition of the $c_i(\T)$ elements in column $2i+1$ (resp. the $r_i(\T)$ elements in column $2i+1$) has length $a_i$ (resp. $b_i$), for $i=1,\ldots,\ell(\T)$.
Then $\ell(\pi_1) = \ell(\T) + \sum_i a_i$ and $\ell(\pi_2) = \ell(\T) + \sum_i b_i$, and $\ell(\pi_1|K) = \ell(\T)$, so $d = \ell(\T) + \sum_i a_i + \sum_i b_i -1$. There are $S(c_i(\T),a_i)\, a_i!$ (resp.  $S(r_i(\T),b_i)\, b_i!$) such ordered set partitions for each $i$, from which the result follows.
\end{proof}

For fixed $k$ and $\ell$, there are ${n \choose k}$ choices for a set $K$ of $k$ elements, there are $\ell! \, S(k,\ell)$ choices for an ordered set partition $K_1|\cdots|K_\ell$ of $K$ of size $\ell$, and there are $(2\ell+1) + 2(2\ell-1) + \cdots + 2(3) + 2(1) = 2 \ell^2 + 2 \ell + 1$ choices for where to place each element not in $K$ in the harmonic table. Therefore the number of harmonic tables for $[n]$ is
\[
\sum_{k=1}^{n-1} \sum_{\ell = 1}^k {n \choose k} S(k, \ell) \, \ell! \, (2 \ell^2 + 2 \ell + 1)^{n-k}.
\]
Using Proposition \ref{prop:fvector} one can compute the $f$-vector of the first few harmonic polytopes:
\begin{eqnarray*}
f(H_{1,1}) &=& (1,1), \\
f(H_{2,2}) &=& (1,6,6,1), \\
f(H_{3,3}) &=& (1,66,144,102,24,1), \\
f(H_{4,4}) &=& (1,1200,4008,5124,3072,834,78,1).
\end{eqnarray*}

\subsection{The harmonic polytope and the bipermutohedron}

As stated in the introduction, the harmonic polytope is one of two polytopes that arose in Ardila, Denham, and Huh's work on the Lagrangian geometry of matroids. The other one is the \emph{bipermutohedron}. We now describe the combinatorial relationship between them. Only for this subsection, we assume familiarity with the construction of the bipermutohedral fan $\Sigma_{n,n}$ and the bipermutahedron $\Pi_{n,n}$ in \cite[Section 2]{ADH1}.

We have shown that the harmonic polytope has $3^n-3$ facets and $(n!)^2(1+\frac12+ \cdots + \frac1n)$ vertices. In turn, the bipermutohedron has $3^n-3$ facets and $(2n)!/2^n$ vertices.
The harmonic polytope is a Minkowski summand of (a multiple of) the bipermutohedron, as shown by the following proposition, originally discovered in \cite[Proposition 2.11]{ADH1}. 
We give an alternative proof that makes the combinatorial relationship between these objects more explicit.

\begin{proposition}
The harmonic fan is a coarsening of the bipermutohedral fan.
\end{proposition}

\begin{proof}
Suppose a point $(z,w) \in \N_n \times \N_n$ is in the interior of cone $\sigma_{\B}$ of the bipermutohedral fan $\Sigma_{n,n}$, corresponding to a bisequence $\B$. Then $z_k+w_k$ is minimized precisely for the set $K$ of indices $k \in [n]$ that appear only once in $\B$. This places the point $(z,w)$ in the chart $\mathscr{C}_k$ of the bipermutohedral fan for each $k \in K$. Fix one such $k$.

Now, as explained in \cite[Proposition 2.9]{ADH1}, the order of the first occurrences of each $i \in [n]$ in the bisequence $\B$ is determined by the reverse order of the numbers $Z_i = z_i - z_k$, which is the reverse order of the $z_i$s. 
Similarly, the order of the second occurrences of each $i \in [n]$ in the bisequence $\B$ is determined by the reverse order of the numbers $W_i = w_k-w_i$, which is the order of the $w_i$s.

We conclude that $(z,w)$ is in the interior of the cone of the harmonic fan  indexed by the harmonic triple $(K; \pi_1, \pi_2)$ where $K$ is the set of elements of $[n]$ appearing only once in $\B$, 
$\pi_1$ is the ordered set partition obtained from the order of the first occurrence of each $i$ in $\B$, and
$\pi_2$ is the ordered set partition obtained by reversing the order of the second  occurrence of each $i$ in $\B$.
\end{proof}

For example, if $(z,w)$ is in the interior of the cone of the bipermutohedral fan $\Sigma_{6,6}$ indexed by the bisequence $\B=34|2|356|1|247|6$, then $(z,w)$ is in the interior of the cone of the harmonic fan $\NN(H_{6,6})$ indexed by the harmonic triple $\tau=(\textbf{157}; \, 34|2|\textbf{5}6|\textbf{1}|\textbf{7}, \, 6|24\textbf{7}|\textbf{1}|3\textbf{5})$.

\section{The volume of $H_{n,n}$} \label{sec:volume}

The goal of this section is to compute the volume of the harmonic polytope. As stated in the introduction, computing the volume of an arbitrary polytope is a very difficult task; we need to use in an essential way the combinatorial structure of our polytope. Our computation relies on the theory of mixed volumes and the Bernstein-Khovanskii-Kushnirenko Theorem which relates these volumes to the enumeration of solutions of systems of polynomial equations. It also relies on the theory of toric edge ideals to enumerate those solutions.
Before reviewing the basics of this theory, let us comment on the definition of volume used here.

\bigskip

\noindent \textbf{\textsf{Normalizing the volume.}}
Most of the polytopes that we study are not full-dimensional in their ambient space, and we need to define their volumes and mixed volumes carefully. Let $P$ be a $d$-dimensional polytope on an affine $d$-plane $L \subset \Z^n$. Assume $L \cap \Z^n$ is a lattice translate of a $d$-dimensional lattice $\Lambda$. We call a lattice $d$-parallelotope in $L$ \emph{primitive} if its edges generate the lattice $\Lambda$; all primitive parallelotopes have the same volume. Then we define the \emph{normalized volume} of a $d$-polytope $P$ in $L$ to be $\vol(P) := \textsf{EVol}(P)/\textsf{EVol}(\square)$ for any primitive parallelotope $\square$ in $L$, where $\textsf{EVol}$ denotes Euclidean volume. By convention, the normalized volume of a point is $1$.
Throughout the paper, all volumes and mixed volumes are normalized in this way.

\subsection{Mixed volumes and Bernstein-Khovanskii-Kushnirenko's Theorem }
\label{subsec:volume}
  
\begin{theorem} (McMullen, \cite{McMullen})
There is a unique function $\mv(Q_1, \ldots, Q_d)$ defined on $d$-tuples of polytopes in $\R^d$, called the \emph{mixed volume} such that for any collection of polytopes $P_1, \ldots, P_m$ in $\R^d$ and any nonnegative real numbers $\lambda_1, \ldots, \lambda_m$, we have
\begin{equation}\label{eq: mixed vol def}
\vol(\lambda_1P_1+ \cdots + \lambda_m P_m) = \sum_{i_1, \ldots, i_d} \mv(P_{i_1}, \ldots, P_{i_d}) \lambda_{i_1} \cdots \lambda_{i_d},
\end{equation}
summing over all ordered $d$-tuples $(i_1, \ldots, i_d)$ with $1 \leq i_k \leq m$ for $1\leq k \leq d$.
Moreover, the function $\mv(Q_1, \ldots, Q_d)$ is symmetric; that is, $\mv(Q_1, \ldots, Q_d) = 
\mv(Q_{\sigma(1)}, \ldots, Q_{\sigma(d)})$ for any permutation $\sigma$ of $[d]$.
\end{theorem}

Mixed volumes have the following algebraic interpretation.

\begin{theorem}[Bernstein-Khovanskii-Kushnirenko Theorem]\cite{Berstein}\label{thm:BKK}
Let $A_1,\ldots,A_d\subset \Z^d$ be $d$ finite sets of lattice points, and let $Q_i=\conv(A_i)$ for $i=1,\ldots,d$.
If the number of solutions in the torus $(\C^*)^d$ to the system
	$$
	\begin{cases}
	\displaystyle \sum_{\alpha\in A_1} \lambda_{1,\alpha} x^\alpha=0,\\
	\qquad\qquad\quad\,\,\,  \vdots\\
	\displaystyle \sum_{\alpha\in A_d} \lambda_{d,\alpha} x^\alpha=0
	\end{cases}
	$$
is finite for a given choice of complex coefficients $\lambda_{i,\alpha}$, then that number is bounded above by $d! \,\mv(Q_1,\ldots,Q_d)$.
Moreover, if the coefficients $\lambda_{i,\alpha}$ are sufficiently generic, the number of solutions equals $d! \,\mv(Q_1,\ldots,Q_d)$.
\end{theorem}

The BKK Theorem is most often used to count or bound the solutions to a system of polynomial equations by computing the corresponding mixed volume. As Postnikov showed in his computation of volumes of generalized permutahedra \cite{Postnikov}, it can also be used in the reverse direction. This will be our approach as well: we will compute mixed volumes by counting the solutions to the associated systems of polynomial equations. This is seldom possible. In Postnikov's case it is easy because he obtains systems of linear equations, which have 0 or 1 solutions. In our case it is also possible, though new ideas are needed. Our systems of equations are not linear, but they are bilinear, and this allows us to express the resulting enumeration problems in terms of the combinatorics of toric ideals of graphs and the enumeration of lattice points in polytopes.

\medskip

To apply this general discussion to the harmonic polytope, we begin by defining the segments
\[
\Delta_{ij} := \, \conv \{\e_i ,\e_j  \} 
\qquad \text{ and }\qquad
\Delta_{\bi\bj} := \, \conv \{\f_i ,\f_j  \}
\qquad \text{ for } 1 \leq i < j \leq n.
\]
The permutohedron equals $\displaystyle \Pi_n = \e_{[n]} + \sum_{i<j}\Delta_{ij}$  \cite[Proposition 2.3]{Postnikov} so
the harmonic polytope equals
\begin{equation} \label{eq:Minkowski}
H_{n,n}= \e_{[n]} + \f_{[n]} + \sum_{i<j}\Delta_{ij}+\sum_{i<j}\Delta_{\bar{i}\bar{j}}+D_n 
\quad 
\subset 
\quad
\R^n \times \R^n. 
\end{equation}
The first two summands $\e_{[n]}$ and $\f_{[n]}$ simply introduce translations, so we focus on the remaining ones.
Given graphs $G$ and $G'$ on vertex set $[n]$ with edge multisets $\{i_1j_1, \ldots, i_rj_r\}$ and $\{\bi_1\bj_1, \ldots, \bi_s \bj_s\}$, respectively, we define their mixed volume to be
\begin{equation} \label{eq:mixed}
 \mv(G, G') := \mv(\Delta_{i_1j_1},\ldots,\Delta_{i_rj_r},\Delta_{\bi_1\bj_1},\ldots,\Delta_{\bi_s\bj_s},\underbrace{D_{n},\ldots,D_{n}}_{k \text{ times}}) 
\end{equation}
where $k=2n-2-r-s$. We also let ${2n-2 \choose G,G'; D_n}$ denote the number of distinct permutations of the sequence $(\Delta_{i_1j_1},\ldots,\Delta_{i_rj_r}, \Delta_{\bi_1\bj_1}, \ldots, \Delta_{\bi_s\bj_s},D_{n},\ldots,D_{n})$. Combining \eqref{eq: mixed vol def} with the fact that mixed volumes are symmetric, we obtain:

\begin{equation}\label{eq: volume expansion}
\vol(H_{n,n})=\sum_{G, G'} \binom{2n-2}{G,G'; D_n} \mv(G, G'),
\end{equation}
summing over all pairs of graphs $G$ and $G'$ on $[n]$. Therefore it remains to compute the mixed volumes $\mv(G, G')$.

\begin{remark} 
The normalized volume $\vol(H_{n,n})$ is equal to the Euclidean volume of the projection of $H_{n,n}$ onto $\Z^{\{2,\ldots,n\}} \times \Z^{\{2,\ldots,n\}}$.
This projection of $H_{n,n}$ is equal to the Minkowski sum of the images under this projection of the polytopes appearing in \eqref{eq:Minkowski}. 
Thus  $\mv(G, G')$ is the mixed volume of the projections of $\Delta_{i_1j_1},\ldots,\Delta_{i_rj_r}, \Delta_{\bi_1\bj_1}, \ldots, \Delta_{\bi_s\bj_s}, D_n, \ldots, D_n$ onto $\Z^{\{2,\ldots,n\}} \times \Z^{\{2,\ldots,n\}}$.

The BKK Theorem then tells us that $(2n-2)! \, \mv(G, G')$ counts the solutions in $(\C^*)^n \times (\C^*)^n$ to the following system of equations:
	\begin{equation} \label{eq:E(G,G')}
	\E(G,G'): \quad 
	\begin{cases}
	x_{i} = \lambda_{ij}x_{j}, \text{ for } ij\in E(G) 
	\qquad\qquad
	&\nu_{11}x_{1}y_{1}+\cdots+\nu_{1n}x_{n}y_{n}=0 \\
	y_{i} = \mu_{ij}y_{j}, \text{ for } \bi\bj\in E(G')  &\qquad\qquad\qquad\qquad \qquad\quad     \vdots\\
	x_1=y_1=1 &\nu_{k1}x_{1}y_{1}+\cdots+\nu_{kn}x_{n}y_{n}=0,
	\end{cases}
	\end{equation}
where $G$ and $G'$ have $r$ and $s$ edges respectively, $k := 2n-2-r-s$, and the coefficients $\lambda_{ij}, \mu_{ij}, \nu_{ij}$ are chosen generically.
\end{remark}

\subsection{Mixed volumes, toric ideals, root polytopes, and trimmed generalized permutahedra } \label{subsec:mixed}
 
In this section we compute the mixed volumes \eqref{eq:mixed} of the harmonic polytope. We begin by showing that most of them vanish. 

\begin{lemma}\label{lem: forests are needed}
If $G$ or $G'$ contains a cycle then the mixed volume $\mv(G, G') = 0$.
\end{lemma}

\begin{proof}
By \Cref{thm:BKK}, $(2n-2)! \, \mv(G, G')$ counts the solutions in $(\C^*)^n \times (\C^*)^n$ to the system of equations $\E(G,G')$.
Suppose that $G$ contains the cycle
	$$
	i_1\to i_2\to\cdots\to i_\ell\to i_1.
	$$
for some vertices $i_1, \ldots, i_\ell \in [n]$. The equations of the corresponding $k$ edges
	$$
	x_{i_1} = \lambda_{i_1i_2}x_{i_2}\qquad \ldots\qquad x_{i_{k-1}} = \lambda_{i_{k-1}i_k}x_{i_k}\qquad x_{i_k} = \lambda_{i_ki_1}x_{i_1}
	$$
imply that $x_{i_1} = 	 (\lambda_{i_1i_2} \cdots \lambda_{i_{k-1}i_k} \lambda_{i_ki_1})x_{i_1}$. Since the $\lambda_{ij}$s are chosen generically, the only solution to this equation is $x_{i_1}=0$. It follows that the system of equations $\E(G,G')$ has no solutions in the torus $(\C^*)^n \times (\C^*)^n$, and $\mv(G, G') = 0$.
\end{proof}

Our next goal is to describe the non-zero mixed volumes $\mv(G, G')$. To accomplish it, we will 
require some additional constructions.

\bigskip

\noindent \textbf{\textsf{The Bipartite Graph and the Root Polytope.}}
Fix graphs $G$ and $G'$ on $[n]$. Let $\I=\{I_1,\ldots,I_p\}$ and $\J=\{J_1 , \ldots , J_q\}$ be the set partitions of $[n]$ into connected components of $G$ and $G'$, respectively. Let $I(k)$ and $J(k)$ denote the parts of $\I$ and $\J$ containing vertex $k$ for $k \in [n]$.
Define the bipartite graph $\Gamma = \Gamma_{\I,\J}$ with vertex set 
$\I\cup \J$ and  $n$ edges $I(k)J(k)$ for $1 \leq k \leq n$. This graph may have several edges connecting the same pair of vertices. We give the edge $I(k)J(k)$ the label $k$. Notice that the label of a vertex in $\Gamma$ is just the set of labels of the edges containing it. Therefore we can remove the vertex labels, and simply think of $\Gamma$ as a bipartite multigraph on edge set $[n]$.

The \emph{edge polytope} of $\Gamma$ is 
	\begin{eqnarray*}
	R_{\Gamma} &:=&  \conv\{\e_{I_a}+\f_{J_b} \, : \,  I_a \in \I, \, J_b \in \J, \, I_a\cap J_b\neq\varnothing\} \\
	&=& \conv\{\e_{I(1)} + \f_{J(1)}, \ldots, \e_{I(n)} + \f_{J(n)}\}  \subset \R^p \times \R^q,
	\end{eqnarray*}
writing $\e_{I(1)}, \ldots, \e_{I(p)}$ and $\f_{J(1)}, \ldots, \f_{J(p)}$ for the standard bases of $\R^p \cong \R^\I$ and $\R^q \cong \R^\J$. This polytope lives on the codimension 2 subspace cut out by the equations $x_1 + \cdots + x_p = y_1 + \cdots + y_q = 1$.
	
\begin{example} \label{ex: graph}
Consider the following graphs
	$$
	
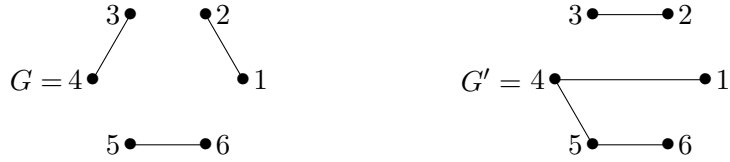
\begin{tikzpicture}
	\draw (60:1)node{$\bullet$} node[right]{2}--(0:1)node{$\bullet$} node[right]{1}
		(180:1)node{$\bullet$} node[left]{4}--(120:1)node{$\bullet$} node[left]{3}
		(300:1)node{$\bullet$} node[right]{6}--(240:1)node{$\bullet$} node[left]{5}
		(-1.75,0) node {$G=$}
		;
	\end{tikzpicture}
	\qquad\qquad\qquad
	\begin{tikzpicture}
	\draw(5*60:1)node{$\bullet$} node[right]{6} -- (4*60:1)node{$\bullet$} node[left]{5}-- (3*60:1)node{$\bullet$} node[left]{4}--(0:1)node{$\bullet$} node[right]{1}
		(60:1)node{$\bullet$} node[right]{2}--(2*60:1)node{$\bullet$} node[left]{3}
		(-1.85,0) node {$G'=$}
		;
	\end{tikzpicture}
	.$$
The partitions corresponding to these graphs are $\I=\{12,34,56\}$ and $\J=\{1456,23\}$, omitting brackets for easier readability. The 
associated bipartite multigraph is
	$$
	
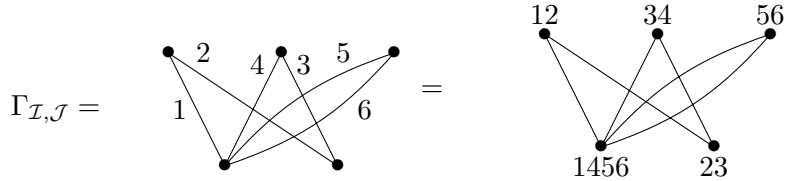
\begin{tikzpicture}[scale=1.5]
	\draw (-1,.5) node {$\Gamma_{\I,\J}=$}
		(1.5,0)node{$\bullet$}	node[below]{}	--(0,1)node{$\bullet$}	node[above]{}	--	(0.5,0)node{$\bullet$}	node[below]{}--	(1,1)node{$\bullet$}	node[above]{}	--	(1.5,0)
		(0.5,0) to [bend left=15] (2,1)node{$\bullet$}	node[above]{}
		(0.5,0) to [bend right=15] (2,1)
		;
	\draw (.1,.5) node {1}
		(.31,1) node {2}
		(.79,.9) node {4}
		(1.2,.9) node {3}
		(1.55,1) node {5}
		(1.74,.5) node {6}
		;
	\end{tikzpicture}
	\begin{tikzpicture}[scale=1.5]
	\draw 
		(-1,.5) node {$=$}
		(1.5,0)node{$\bullet$}	node[below]{23}	--	(0,1)node{$\bullet$}	node[above]{12}	--	(0.5,0)node{$\bullet$}	node[below]{1456}--	(1,1)node{$\bullet$}	node[above]{34}	--	(1.5,0)
		(0.5,0) to [bend left=15] (2,1)node{$\bullet$}	node[above]{56}
		(0.5,0) to [bend right=15] (2,1)
		;
	\end{tikzpicture}
	$$
and the corresponding edge polytope is
\[
R_\Gamma = \conv(\e_a + \f_A, \, \e_a + \f_B, \,  \e_b + \f_B, \, \e_b + \f_A, \, \e_c + \f_A, \, \e_c + \f_A) \subset \R^{abc} \times \R^{AB},
\]
writing $a=12, b=34, c=56$ and $A=1456, B=23$.
\end{example}

\begin{lemma} \label{lem:latticepts}
The only lattice points of the edge polytope $R_{\Gamma}$ are its vertices.
\end{lemma}

\begin{proof}
The polytope $R_{\Gamma}$ is contained in the sphere $S$ centered at the origin with radius $\sqrt{2}$, so it can only contain lattice points of norm $0, 1,$ or $\sqrt{2}$. Since $R_{\Gamma}$ lies on the hyperplanes $\sum_i x_i = 1$ and $\sum_i y_i = 1$, it cannot contain a lattice point of norm $0$ (the origin) or $1$ (a point of the form $\pm \e_i$ or $\pm \f_i$). Therefore every lattice point in $R_{\Gamma}$ must be of the form $\e_i + \f_j$ for some $i, j \in [n]$. These points are all on the surface of the sphere $S$, so they are in convex position; therefore, if a point $\e_i+\f_j$ is in $R_{\Gamma}$, it must in fact be a vertex of $R_{\Gamma}$. The result follows.
\end{proof}

A lattice polytope $P$ is \emph{normal} if for all positive integers $k$ and all lattice points $x$ in $kP$ there exist lattice points $x_1, \ldots, x_k$ in $P$ such that $x = x_1 + \cdots + x_k$. 
A lattice polytope $P$ is \emph{very ample} if the above property holds for all sufficiently large integers $k$. 
This is a favorable property algebro-geometrically, because if $P$ is very ample then the lattice points of $P$ provide a concrete projective embedding of the toric variety $X_P$ of $P$, as follows.
Let $P\cap \Z^d=\{\mathsf{a}_1,\ldots, \mathsf{a}_s\} =:A$.
The projective embedding of $X_P$ is the Zariski closure of the image of the map
\begin{eqnarray*}
 (\C^*)^d & \longrightarrow & \C\mathbb{P}^{s-1} \\
	 \mathbf{t} & \longmapsto& ( \mathbf{t}^{\mathsf{a}_1},\ldots, \mathbf{t}^{\mathsf{a}_s})
\end{eqnarray*}
and its defining ideal $I_A$ is the kernel of the homomorphism
\begin{eqnarray*}
\varphi : \C[x_1, \ldots, x_s] & \longrightarrow & \C[t_1^{\pm 1}, \ldots, t_d^{\pm 1}] \\
x_i & \longmapsto & \mathbf{t}^{\mathsf{a}_i}.
\end{eqnarray*}
 The map above induces the map of lattices
\begin{eqnarray*}
\widehat{\varphi} : \Z^s & \longrightarrow & \Z^d \\
	\e_i & \longmapsto& \mathsf{a}_i,
\end{eqnarray*}
where $\e_1,\ldots,\e_s$ is the standard basis of $\Z^s$.
The kernel of ${\varphi}$ is the \emph{toric ideal}
\[
 I_A = \langle \mathbf{x}^{\mathsf{u}} - \mathbf{x}^{\mathsf{v}} \, : \, \mathsf{u}, \mathsf{v} \in \mathbb{N}^s, \, \widehat{\varphi}(\mathsf{u}) = \widehat{\varphi}(\mathsf{v}) \rangle \,\, \subset \,\, \C[x_1, \ldots, x_s];
\]
see \cite[\S2.1 and \S2.3]{Cox-Little-Schenck}.

\begin{proposition}\label{lem:normal}
If $\Gamma$ is bipartite, the edge polytope $R_{\Gamma}$ is normal.
\end{proposition}

\begin{proof}
Let
\[
C_{\Gamma} = \textsf{cone}(R_{\Gamma}) = \{\lambda \, q \, : \, q \in R_{\Gamma}, \lambda \geq 0\} \subset \R^n \times \R^n
\]
be the cone over the polytope $R_{\Gamma}$. 
Consider a lattice point $x$ in $k \,R_{\Gamma}$. The cone $C_{\Gamma}$ is generated by the vertices of $R_{\Gamma}$, so $x$ is a positive combination of them. By Caratheodory's theorem, $x$ can be expressed a positive combination of only $e$ linearly independent vertices of $R_{\Gamma}$, say $v_1, \ldots, v_e$, for some $e \leq \dim \, R_{\Gamma}$. But the vector configuration $\{\e_i + \f_j \, : \, 1 \leq i, j \leq n\}$ is unimodular, so $v_1, \ldots, v_e$ form a lattice basis for $\textsf{cone}(v_1, \ldots, v_e) \cap (\Z^n \times \Z^n)$. It follows that $x$ is a positive \textbf{integer} combination of $v_1, \ldots, v_e \in R_{\Gamma}$. We conclude that $R_{\Gamma}$ is normal as desired.
\end{proof}

\bigskip

\noindent \textbf{\textsf{The Toric Ideal, the Toric Variety, and the Trimmed Generalized Permutahedra.}}
The graph $\Gamma = \Gamma_{\I,\J}$ gives rise to a ring homomorphism
\begin{eqnarray*}
\R[z_e \, : e \text{ edge of } \Gamma] &\longrightarrow& \R[y_v \, : v \text{ vertex of } \Gamma] \\
z_e & \longmapsto & y_iy_j  \quad \text{ where edge $e$ joins vertices $i$ and $j$}
\end{eqnarray*}
The kernel of this homomorphism is called the \emph{toric ideal} $I_\Gamma$ of $\Gamma$; it is a homogeneous ideal given by the cycles of even length in $\Gamma$:
\[
I_\Gamma = \langle z_{e_1} z_{e_3} \cdots z_{e_{2k-1}} - z_{e_2} z_{e_4} \cdots z_{e_{2k}} \, : \, e_1e_2\cdots e_{2k} \text{ is a cycle of } \Gamma \rangle;
\]
see \cite[Section 5.3]{Herzog-Hibi-Ohsugi}. This ideal is related to the edge polytope as follows.

\begin{proposition}\label{prop:toricvariety}
If $\Gamma$ is a bipartite graph, the projective variety of the toric ideal $I_{\Gamma}$ is an embedding of the toric variety $X_\Gamma$ of the edge polytope $R_\Gamma$. 
\end{proposition}

\begin{proof}
This holds thanks to Lemma \ref{lem:latticepts} and \ref{lem:normal}; see \cite[\S 2.3]{Cox-Little-Schenck}.
\end{proof}

The following polytopes will also play an important role. Consider the Minkowski sums
\[
P_{\Gamma}:=\sum_{i=1}^p \Delta_{\nbr(I_i)}\subset\R^{q} 
\qquad \text{and} \qquad 
Q_{\Gamma}:=\sum_{j=1}^q \Delta_{\nbr(J_j)}\subset\R^{p}
\]
where $\Delta_{I}:=\conv\{\e_{i} \, : \, i\in I\}$, and where
$\nbr(I_i) = \{j \in [q] \, : \, I_iJ_j \text{ is an edge of } \Gamma\}$ and 
 $\nbr(J_j) = \{i \in [p] \, : \, I_iJ_j \text{ is an edge of } \Gamma\}$
 denote
 the neighborhoods of $I_i$ and $J_j$ in $\Gamma$.
Finally, define the \emph{trimmed generalized permutahedra} of $\Gamma$ to be the Minkowski differences
\[
P^{-}_{\Gamma}:=P_{\Gamma}-\Delta_{[q]}\subset\R^{q} 
\qquad \text{and} \qquad   
Q^{-}_{\Gamma}:=Q_{\Gamma}-\Delta_{[p]}\subset\R^{p}
\]

\begin{example} \label{ex: ideal}
We return to Example \ref{ex: graph}. The toric ideal of $\Gamma$ is
\[
I_\Gamma = \langle z_1z_3-z_2z_4, z_5-z_6 \rangle \subset \C[z_1,z_2,z_3,z_4,z_5,z_6].
\]
The generalized permutahedra associated to $\Gamma$ are
\[
	P_{\Gamma}=\Delta_{abc}+\Delta_{ab}\subset\R^{abc}
		\qquad \text{and} \qquad   
	Q_{\Gamma}=2\Delta_{AB}+\Delta_{A}	\subset\R^{AB} \\
\]
and the trimmed generalized permutahedra are
\[
	P^{-}_{\Gamma}=\Delta_{ab}	\subset\R^{abc} 
		\qquad \text{and} \qquad   
	Q^{-}_{\Gamma}=\Delta_{AB}+\Delta_A	\subset\R^{AB} .
\]
In general, the polytopes $P^{-}_{\Gamma}$ and $Q^{-}_{\Gamma}$ live in different dimensions and can be very different from each other. However, we will see that they always have the same number of lattice points.
\end{example}

\bigskip

\noindent \textbf{\textsf{Putting it all Together.}} We now have all the ingredients to describe the mixed volumes $\mv(G,G')$.

\begin{proposition}\label{key lemma}
Let $G$ and $G'$ be acyclic graphs on $[n]$ and $\Gamma$ be the corresponding bipartite graph, having $p$ and $q$ vertices on each side of the bipartition. 
The following numbers are equal:
\begin{enumerate}
\item
The $(2n-2)$-dimensional mixed volume $\mv(G, G')$ multiplied by $(2n-2)!$.
\item
The $(p+q-2)$-dimensional volume of the edge polytope $R_{\Gamma}$ multiplied by $(p+q-2)!$.
\item
The number $i(P^{-}_{\Gamma})$ of lattice points in the trimmed generalized permutahedron $P^{-}_{\Gamma}$ in $\R^q$.
\item
The number $i(Q^{-}_{\Gamma})$ of lattice points in the  trimmed generalized permutahedron $Q^{-}_{\Gamma}$ in $\R^p$.
\end{enumerate}
Furthermore, the numbers above are zero if and only if $\Gamma$ is disconnected. 
\noindent
If $\Gamma$ is connected, the numbers above are equal to:
\begin{enumerate}
 \item[5.] the degree of the projective embedding $V(I_\Gamma)$ of the toric variety $X_{\Gamma}$.
 \end{enumerate}
\end{proposition}

Recall that all volumes are normalized so the volume of a primitive parallelotope in any dimension is $1$.

\begin{proof}
Let $\E(G,G')$ be the system of equations \eqref{eq:E(G,G')} associated to the mixed volume $\mv(G,G')$. By \Cref{thm:BKK}, the quantity in 1. counts the solutions in $(\C^*)^n \times (\C^*)^n$ to $\E(G,G')$.

\bigskip

\noindent
(1. = 2.)
The set of solutions to $\E(G,G')$ is the variety $V(I_{G,G'}+J)$, where 
\begin{eqnarray*}
I_{G,G'} &:=&
	\big\langle x_i-\lambda_{ij}x_j\, : \, i<j, \, ij\in E(G)\big\rangle 	+
	\big\langle y_i-\mu_{ij}y_j\, : \,  i<j, \, ij\in E(G')\big\rangle \\
J &:=& 
	\big\langle \nu_{i1}x_{1}y_{1}+\cdots+\nu_{in}x_{n}y_{n} \, : \, 1 \leq i  \leq k\big\rangle + \langle x_1-1, y_1-1 \rangle
\end{eqnarray*}
in $\C[x_1^\pm,\ldots,x_n^\pm,y_1^\pm,\ldots,y_n^\pm]$. 

Consider the subspace 
	$$
	\L=\{(x_1,\ldots,x_n,y_1,\ldots,y_n)\, : \, x_{i} = \lambda_{ij}x_{j} \text{ for } ij\in E(G)\, , \,
	y_{i} = \mu_{ij}y_{j} \text{ for } ij\in E(G')\} \subset \C^n \times \C^n
	$$
and the projection
	\begin{eqnarray*}
	\widetilde\psi: \L 	& \longrightarrow&  \C^p\times\C^q\\
	(x_1,\ldots,x_n,y_1,\ldots,y_n) 	&\longmapsto& (x_{\min I_1},\ldots,x_{\min I_p},y_{\min J_1},\ldots,y_{\min J_q}).
	\end{eqnarray*}	
Note that for $(x,y) \in \L$ if we have $x_{\min I_a}=0$
then $x_{i}=0$ for all $i\in I_a$, since $I_a$ is a connected component in $G$;
the same holds for the $y$s.
Therefore $\widetilde\psi$ is injective and, since $\dim(\L)=p+q=\dim(\C^p \times \C^q)$, it follows that $\widetilde\psi$ is an isomorphism of affine varieties.
Moreover, $x_{\min I_a}=0$ if and only if $x_{i}=0$ for some $i\in I_a$; the same holds for the $y$s. This implies that the restriction of $\widetilde\psi$ to $ \L\cap \big((\C^*)^n \times (\C^*)^n\big)$ is a morphism with image $(\C^*)^p \times (\C^*)^q$.
This morphism defines the following isomorphism between the coordinate rings of $ \L\cap \big((\C^*)^n \times (\C^*)^n \big) $ and $(\C^*)^p \times (\C^*)^q$:
	\begin{eqnarray*}
	\psi:\C[x_{I_1}^\pm,\ldots,x_{I_p}^\pm,y_{J_1}^\pm,\ldots,y_{J_q}^\pm]   & \longrightarrow & \C[x_1^\pm,\ldots,x_n^\pm,y_1^\pm,\ldots,y_n^\pm]/I_{G,G'}\\
	x_{I_a} &\longmapsto& \bar{x}_{\min I_a}\\
	y_{J_b} &\longmapsto& \bar{y}_{\min J_b}.
\end{eqnarray*}
Let $\overline{J}$ be the image of $J$ in the quotient $\C[x_1^\pm,\ldots,x_n^\pm,y_1^\pm,\ldots,y_n^\pm]/I_{G,G'}$.
By Noether's isomorphism theorems we have
\begin{align*}
	 \C[x_1^\pm,\ldots,x_n^\pm,y_1^\pm,\ldots,y_n^\pm]/(I_{G, G'}+J)
	 & \cong
	 \left(	\C[x_1^\pm,\ldots,x_n^\pm,y_1^\pm,\ldots,y_n^\pm]/	I_{G,G'}	\right) \big/	\big(	(I_{G,G'}+J)/I_{G,G'}	\big)\\
	&\cong 
	\C[x_{I_1}^\pm,\ldots,x_{I_p}^\pm,y_{J_1}^\pm,\ldots,y_{J_q}^\pm] \big/  \psi^{-1}\big((I_{G,G'} + J)/I_{G,G'} \big)\\
	&=
	\C[x_{I_1}^\pm,\ldots,x_{I_p}^\pm,y_{J_1}^\pm,\ldots,y_{J_q}^\pm] / \psi^{-1}(\overline{J}).
	\numberthis \label{eq:iso}
	\end{align*}

Note that for $1 \leq m \leq n$ we have $\bar{x}_m = \lambda \, \bar{x}_{\min I(m)}$ and $\bar{y}_m = \mu \, \bar{y}_{\min J(m)}$ for nonzero scalars $\lambda$ and $\mu$. Thus we have, for $i\in[k]$,
$$
\psi^{-1}(\nu_{i1}\bar{x}_{1}\bar{y}_{1}+\cdots+\nu_{in}\bar{x}_{n}\bar{y}_{n})=\eta_{i1}x_{I(1)}y_{J(1)}+\cdots+\eta_{in}x_{I(n)}y_{J(n)}
$$
for some nonzero constants $\eta_{ij}$ that are generic if the $\nu_{ij}s$ are sufficiently generic. 
We conclude that $\psi^{-1}(\overline{J})$ is generated by $k$ generic equations whose Newton polytope is equal to $R_{\Gamma}$, together with $x_{I(1)} - 1$ and $y_{J(1)} - 1$. Recall that $k=2n-2-r-s$ where $r$ and $s$ are the numbers of edges of $G$ and $G'$ respectively. Since these graphs are acyclic, $r=n-p$ and $s=n-q$, so $k=p+q-2$ equals the ambient dimension of the Newton polytope $R_{\Gamma}$.

By the BKK Theorem, the  left-hand side of \eqref{eq:iso} is a variety consisting of $(2n-2)! \, \mv(G,G')$ points and the right hand side is a variety consisting of $(p+q-2)! \,  \vol_{p+q-2}(R_\Gamma)$ points. Therefore these two numbers are equal to each other.

\bigskip

\noindent
(2. = 3. = 4.)
In the case that $\Gamma$ is connected, Postnikov 
\cite[Theorem 12.2]{Postnikov} showed that the $(p+q-2)$-dimensional volume of the edge polytope $R_\Gamma$ times $(p+q-2)!$ equals $i(P^-_\Gamma)$ and $i(Q^-_\Gamma)$.

\smallskip

Now assume $\Gamma$ is disconnected. Say $\Gamma = \Gamma_1 \cup \Gamma_2$ where $\Gamma_1$ and $\Gamma_2$ have vertex sets $\I_1 \cup \J_1$ and $\I_2 \cup \J_2$ for $\I_1 \cup \I_2 = \I$ and $\J_1 \cup \J_2 = \J$. 

First observe that $R_\Gamma$ is the convex hull of the union of the edge polytopes $R_{\Gamma_1}$ and $R_{\Gamma_2}$. But these two polytopes have dimension at most $|\I_1| + |\J_1|-2$ and $|\I_2| + |\J_2|-2$, respectively, so $R_\Gamma$ has dimension at most $|\I| + |\J|-4 = p+q-4$, and hence its $(p+q-2)$-dimensional volume is $0$.

On the other hand, 
by the definition of $P_{\Gamma}$,
$$
	P_{\Gamma} \subset \Big\{x\in \R^{q}\, : \, \sum_{j\in \J_1} x_j = |\I_1|, \sum_{j\in \J_2} x_j = |\I_2|	\Big\}.
$$
so $\dim(P_{\Gamma}) < q-1 = \dim(\Delta_{[q]})$. Therefore $P^-_{\Gamma} = P_{\Gamma} - \Delta_{[q]} = \varnothing$ and $i(P^-_\Gamma) = 0$. The proof that $i(Q^{-}_{\Gamma})=0$ is analogous.

\bigskip

We have shown that (1.) = (2.) = (3.) = (4.). We have also shown that if $\Gamma$ is disconnected this number is $0$. On the other hand, if $\Gamma$ is connected, then $\dim R_\Gamma = p+q-2$ by \cite[Lemma 5.4]{Herzog-Hibi-Ohsugi}, so its volume is nonzero.

\bigskip

\noindent
(2. = 5. if $\Gamma$ is connected.) The $(p+q-2)$-dimensional volume of $R_\Gamma$ equals the degree of $V(I_\Gamma)$ by \cite[Theorem 4.16]{Sturmfels}.
\end{proof}

\subsection{An illustrative example}

Let us verify that $(2n-2)! \, \mv(G,G')=i(P^-_{\I,\J})=i(Q^-_{\I,\J})$ for the graphs in \Cref{ex: graph}. This case is small enough that we can do it by hand, and it illustrates the need for the machinery of Section \ref{subsec:mixed}.
Here $n=6$, so $10! \, \mv(G,G')$ is the number of solutions to the system $\E(G,G')$
	$$
	\E(G,G'): \quad
	\begin{cases}
	x_{1} = \lambda_{12} \, x_{2}, \qquad\qquad 
	y_{1} = \mu_{14}\,y_{4}, \qquad \qquad
	\nu_{11}\, x_{1}y_{1}+\cdots+\nu_{16}\, x_{6}y_{6}=0, \\
	x_{3} = \lambda_{34}\,x_{4}, \qquad\qquad 
	y_{2} = \mu_{23}\,y_{3}, \qquad \qquad
	\nu_{21}\,x_{1}y_{1}+\cdots+\nu_{26}\,x_{6}y_{6}=0,  \\
	x_{5} = \lambda_{56}\,x_{6}, \qquad\qquad  
	y_{4} = \mu_{45}\,y_{5}, \qquad \qquad  
	\nu_{31}\,x_{1}y_{1}+\cdots+\nu_{36}\,x_{6}y_{6}=0,\\
	\hspace{3.6cm} y_{5} = \mu_{56}\,y_{6},\\
	x_1=1, \hspace{2.33cm}  y_1=1. \\
	\end{cases}
	$$
for a generic choice of coefficients. The first two columns of $\E(G,G')$ may be rewritten as 
\begin{eqnarray*}
&& 1 = X_{12} := x_1 = \lambda_{12} \, x_2, \quad X_{34}: = x_3 = \lambda_{34} \, x_4, \qquad X_{56}: = x_5 = \lambda_{56}x_6\\
&& 1 = Y_{1456} := y_1 = \mu_{14} \,y_4 = \mu_{14}\mu_{45} \, y_5, = \mu_{14}\mu_{45}\mu_{56} \, y_6, \quad Y_{23} := y_2 = \mu_{23} \, y_3, \quad Y_6:=y_6,
\end{eqnarray*}
so $\E(G,G')$ reduces to the following system of equations:
	$$
	\cH_{\I,\J}: \quad 
	\begin{cases}
	\eta_{11} \,X_{12}Y_{1456}+\eta_{12} \,X_{12}Y_{23}+\eta_{13} \,X_{34}Y_{23}+\eta_{14} \,X_{34}Y_{1456}+\eta_{15} \,X_{56}Y_{1456}+\eta_{16} \,X_{56}Y_{1456}=0,\\
	\eta_{21} \,X_{12}Y_{1456}+\eta_{22} \,X_{12}Y_{23}+\eta_{23} \,X_{34}Y_{23}+\eta_{24} \,X_{34}Y_{1456}+\eta_{25} \,X_{56}Y_{1456}+\eta_{26} \,X_{56}Y_{1456}=0,\\
	\eta_{31} \,X_{12}Y_{1456}+\eta_{32} \,X_{12}Y_{23}+\eta_{33} \,X_{34}Y_{23}+\eta_{34} \,X_{34}Y_{1456}+\eta_{35} \,X_{56}Y_{1456}+\eta_{36} \,X_{56}Y_{1456}=0,\\
	X_{12} = Y_{145} = 1,
	\end{cases}
	$$
where each coefficient $\eta_{ij}$ is obtained by multiplying $\nu_{ij}$ with the $\lambda$s (or their inverses) along a path from $i$ to $\min I(i)$ in $G$ and the $\mu$s (or their inverses) along a path from $j$ to $\min J(j)$ in $G'$. These coefficients are generic if the original $\lambda$s, $\mu$s, and $\nu$s are sufficiently generic.
This reduction of $\E(G,G')$ to $\cH_{\I,\J}$ is central to the proof of (1. $\Longleftrightarrow$ 2.) in Proposition~\ref{key lemma}.

If we write
\[
z_1 = X_{12}Y_{1456} = 1, \quad 
z_2 = X_{12}Y_{23}, \quad
z_3 = X_{34}Y_{23}, \quad
z_4 = X_{34}Y_{1456}, \quad
z_5 = X_{56}Y_{1456}, \quad
z_6 = X_{56}Y_{1456}
\]
we get a generic system of 3 equations in 5 unknowns $z_2, \ldots, z_6$. Solving this system, we obtain an expression for each of $z_2, \ldots, z_4$ as a linear function of $z_5$ and $z_6$. 
Now, the $z_i$s satisfy two equations 
\[
z_1z_3=z_2z_4, \qquad z_5=z_6
\]
coming from the two even cycles formed by edges $1,2,3,4$ and edges $5,6$ in $\Gamma$, respectively. Thus $z_2, z_3$ and $z_4$ can be expressed linearly in terms of $z_6$, and the equation $z_1z_3=z_2z_4$ turns into a quadratic equation satisfied by $z_6$, which has $2$ solutions. Reversing the steps of our computation, we obtain $2$ solutions to the original system. We conclude that $10! \, \mv(G,G') = 2$.
This agrees with the fact that 
$P^{-}_{\Gamma}=\Delta_{ab} \subset\R^{abc}$ and 
$Q^{-}_{\Gamma}=\Delta_{AB}+\Delta_A	\subset\R^{AB}$
each contain two lattice points.

The procedure above works for general acyclic graphs $G$ and $G'$ such that $\Gamma$ is connected; the relations among the $z_i$s are precisely given by the toric ideal $I_\Gamma$. The last step of the computation cannot be done by hand in general; instead, one needs to know the degree of $I_\Gamma$. We find it by computing the number of lattice points in $P^{-}_{\I,\J}$ or in $Q^{-}_{\I,\J}$ -- whichever is easier.

\subsection{The volume}

We are finally ready to prove Theorem \ref{thm:volume}.

\begin{reptheorem}{thm:volume}
The volume of the harmonic polytope is
	\begin{eqnarray*}
	\vol(H_{n,n}) &=& \sum_{\Gamma} \frac{i(P_\Gamma^-)}{(v(\Gamma)-2)!}
	\prod_{v\in V(\Gamma)}\deg(v)^{\deg(v)-2} \\
	&=& \sum_{\Gamma} \frac{\deg(X_\Gamma)}{(v(\Gamma)-2)!}
	\prod_{v\in V(\Gamma)}\deg(v)^{\deg(v)-2}, 
	\end{eqnarray*}
summing over all connected bipartite multigraphs $\Gamma$ on edge set $[n]$. 
Here $i(P_\Gamma^-)$ is the number of lattice points in the trimmed generalized permutahedron $P_\Gamma^-$ of $\Gamma$, $X_\Gamma$ is the projective embedding of the toric variety of $\Gamma$
given by the toric ideal of $\Gamma$, $V(\Gamma)$ is the set of vertices of $\Gamma$, and $v(\Gamma) := |V(\Gamma)|$. 
\end{reptheorem}

\begin{proof}
We use the notation of Sections \ref{subsec:volume} and \ref{subsec:mixed}.
By \eqref{eq: volume expansion}  and \Cref{lem: forests are needed} we have that 
	\begin{align*}
	\vol(H_{n,n})
	&=
	\sum_{\substack{G, G'\\ \text{acyclic}}} \binom{2n-2}{G,G'; D_n} \mv(G,G')\\
	&=
	\sum_{\substack{G, G'\\ \text{acyclic}}}  \frac{(2n-2)!}{k!} \mv(G,G'),
\end{align*}
since the graphs $G$ and $G'$ have no repeated edges. Write $\Gamma$ for the bipartite graph associated to $G$ and $G'$, abusing notation. 
Applying \Cref{key lemma}, and noting that $v(\Gamma)-2 = p+q-2 = k$, it follows that
	\begin{align*}
	\vol(H_{n,n})
	&=
	\sum_{\substack{G, G'\\ \text{acyclic}}}  \vol_{v(\Gamma)-2}(R_\Gamma) \\
	&=
	\sum_{\substack{(G,G')\ \text{acyclic}\\ \text{s.t. } \Gamma \text{ connected}}}
	 \frac{i(P^{-}_{\Gamma})}{(v(\Gamma)-2)!}.
\end{align*}
Since the summands on the right only depend on the partitions $\I,\J$ associated to $G,G'$ we can combine the terms in as follows:
	\begin{align*}
	\vol(H_{n,n})
	&=
	\sum_{\Gamma \text{ connected}}
	 \frac{i(P^{-}_{\Gamma})}{(v(\Gamma)-2)!}
	  \cdot
	  (\text{number of acyclic graphs $G,G'$ whose bipartite graph is $\Gamma$}).
\end{align*}
Now, the edges of the bipartite graph $\Gamma$ determine the labels $\I = \{I_1 , \cdots , I_p\}$ and $\J = \{J_1 , \cdots , J_q\}$ of the vertices of $\Gamma$; we need these to be the partitions of $[n]$ into the connected components of $G$ and $G'$, respectively. 

We specify an acyclic graph $G$ (resp. $G'$) with components $\I$ (resp. $\J$) by specifying, for each $I\in\I$ (resp. $J\in\J$), a tree with $|I|$ (resp. $|J|$) vertices. There are $|I|^{|I|-2}$ (resp. $|J|^{|J|-2}$) such trees. By definition of $\Gamma_{\I,\J}$, $\deg(I)=|I|$ for any $I\in\I$ (and similarly for any $J\in\J$). We thus conclude that
	\begin{eqnarray*}
	\vol(H_{n,n}) 
	&=& \sum_{\Gamma \textrm{ connected}} 
\frac{\deg(X_\Gamma)}{(|V(\Gamma)|-2)!}
	\prod_{v\in V(\Gamma)}\deg(v)^{\deg(v)-2}.
	\end{eqnarray*}
as desired. 
\end{proof}

Using Theorem \ref{thm:volume} one can readily compute the volumes of the first few harmonic polytopes:
\[
\vol(H_{1,1}) = 1, \quad
\vol(H_{2,2}) = 3, \quad
\vol(H_{3,3}) = 33, \quad
\vol(H_{4,4}) = 2848/3.
\]

\section{The number of non-zero mixed volumes}

In this section we compute the number of non-zero mixed volumes of the harmonic polytope, in its Minkowski sum decomposition \eqref{eq:Minkowski}. This is the number of summands that contribute to the volume of the harmonic polytope $H_{n,n}$ in \eqref{eq: mixed vol def}. We do so with the help of the \emph{M\"obius algebra} of the partition lattice, which is denoted $\Pi_n$.\footnote{This should not be confused with the permutohedron, which makes no further appearances in the paper.}

If $\pi = \{B_1, \ldots, B_k\}$ is a set partition of $[n]$, we let $\ell(\pi): = k$ be the number of parts of $\pi$, and
\[
t(\pi) := |B_1|^{|B_1|-2} \cdot  \cdots \cdot  |B_k|^{|B_k|-2}.
\]
Let $\Pi_n$ be the lattice of set partitions of $[n]$ ordered by refinement, so $\sigma \leq \tau$ if every block of $\tau$ is a union of blocks of $\sigma$.

\begin{proposition}\label{prop:summands}
The harmonic polytope  $H_{n,n} =\e_{[n]} + \f_{[n]} +  \sum_{i<j} \Delta_{ij} + \sum_{i<j} \Delta_{\bi \bj} + D_n$ has
\begin{eqnarray*}
a_n &:= & \text{number of pairs of forests $(F_1, F_2)$ on $[n]$ such that $F_1 \cup F_2$ is connected} \\
&=& \sum_{\sigma \in \Pi_n} (-1)^{\ell(\sigma)} (\ell(\sigma)-1)! \Big(\sum_{\tau \leq \sigma} t(\tau)\Big)^2
\end{eqnarray*}
non-zero mixed volumes.
\end{proposition}

\begin{proof}
A non-zero mixed volume cannot involve either of the summands $\e_{[n]}$ or $\f_{[n]}$, since the corresponding equations $\lambda \, x_1 \cdots x_n = 0$ and $\mu \, y_1 \cdots y_n = 0$ have no solutions on the torus for $\lambda$ and $\mu$ generic. Thus we focus on the mixed volumes $\mv(G,G')$.
 
By \Cref{lem: forests are needed} and \Cref{key lemma}, we have that $\mv(G,G')\neq 0$ if and only if $G,G'$ are forests and the associated bipartite graph $\Gamma = \Gamma_{\I,\J}$ is connected. Thus to prove the first statement we will show that $G\cup G'$ is connected if and only if $\Gamma$ is connected.

Suppose that $G\cup G'$ is connected. A path 
	$$
	i_1\to	i_2\to\ldots\to	i_\ell
	$$
in $G\cup G'$ gives rise to a path in $\Gamma$ as follows.
For $j=1,\ldots, \ell$, replace the edge $i_j\to i_{j+1}$ in $G \cup G'$ with the edge $J(i_j)\to I(i_j)= I(i_{j+1})$ if $i_j i_{j+1}\in E(G)$, and with $I(i_j)\to J(i_j)= J(i_{j+1})$ in $\Gamma$ if $i_j i_{j+1}\in E(G')$
Note that the resulting path can easily be modified into a path starting at $I(i_1)$ or $J(i_1)$ by adding or removing the edge $I(i_1)J(i_1)$; a similar modification works for $I(i_\ell)$ or $J(i_\ell)$.
Now, to find a path between any two vertices of $\Gamma$, pick an element of each vertex, construct a path between these elements in $G\cup G'$, and use the procedure above to obtain a path between the desired vertices in $\Gamma$.

Conversely, suppose that $\Gamma$ is connected and consider any two vertices $i, i'$ of $G\cup G'$.
Consider a path
	$$
	P: \qquad I(i_1)\to	J(i_1)=J(i_2)\to	I(i_2)=I(i_3)\to	\ldots\to	J(i_{\ell-1})=J(i_\ell) \qquad \text{ in } \Gamma,
	$$
where $i_1=i$ and $i_\ell=i'$. For each $1 \leq j \leq \ell -1$, we have either $I(i_j) = I(i_{j+1})$ or $J(i_j) = J(i_{j+1})$; since these are connected components of $G$ or $G'$, we can find a path in either $G$ or $G'$ from $i_j$ to $i_{j+1}$. 
We are then able to construct a path in $G\cup G'$ from $i$ to $i'$ by replacing each edge of the path $P$ in $\Gamma$ with a 
path from $i_j$ to $i_{j+1}$ in $G \cup G'$.
This concludes the proof of the first equation.

\smallskip

Now, to choose a pair of forests $(F_1, F_2)$ on $[n]$ such that $F_1 \cup F_2$ is connected, we first choose the set partitions $\pi_1 := \pi(F_1)$ and $\pi_2 = \pi(F_2)$, where $\pi(F)$ denotes the partition of $[n]$ given by the connected components of $F$. 
Notice that $F_1 \cup F_2$ is connected if and only if $\pi_1 \vee \pi_2 = \widehat{1}$ in the partition lattice. Having chosen the partitions $\pi_1$ and $\pi_2$, it simply remains to choose the forests $F_1$ and $F_2$ that give rise to them; there are $t(\pi_1)$ and $t(\pi_2)$ choices for those forests, respectively. It follows that
\[
a_n = \sum_{\substack{\pi_1, \pi_2 \in \Pi_n \\ \pi_1 \vee \pi_2 = \widehat{1}}} t(\pi_1) t(\pi_2).
\]
Now we compute in the \emph{M\"obius algebra} $A(\Pi_n)$ of $\Pi_n$; this is the real vector space with basis $\Pi_n$ equipped with the bilinear multiplication given by the join of the lattice; in symbols,
\[
A(\Pi_n) := \R \, \Pi_n \, \qquad \text{where} \qquad \sigma \cdot \tau := \sigma \vee \tau.
\]
It follows from the definitions that
\begin{equation}
\label{eq:an}
a_n = [\widehat{1}] \, T^2 \qquad \text{ for } \qquad T:= \sum_{\pi \in \Pi_n} t(\pi) \pi,
\end{equation}
where $[\pi] \alpha$ denotes the coefficient of a set partition $\pi \in \Pi_n$ in an element $\alpha \in A(\Pi_n)$, when expressed in the standard basis.

As explained in \cite[Section 3.9]{EC1}, it is useful to define the following elements of the M\"obius algebra $A(\Pi_n)$:
\[
\delta_{\tau} := \sum_{\sigma \geq \tau} \mu(\tau, \sigma) \sigma \qquad \text{ for } \tau \in \Pi_n.
\]
These elements form a basis for $A(\Pi_n)$ because M\"obius inversion tells us that
\[
\tau = \sum_{\sigma \geq \tau} \delta_\sigma, \qquad \text{ for } \tau \in \Pi_n.
\]
Furthermore, they are pairwise orthogonal idempotents:
\[
\delta_\sigma \delta_\tau = 
\begin{cases}
\delta_\sigma & \text{ if } \sigma = \tau, \\
0 & \text{otherwise},
\end{cases}
\]
which makes them very useful for computations in $A(\Pi_n)$. We compute
\begin{eqnarray*}
T &=& \sum_{\tau \in \Pi_n} \Big( t(\tau) \sum_{\sigma \geq \tau} \delta_\sigma \Big)\\
&=& \sum_{\sigma \in \Pi_n} s(\sigma) \,  \delta_\sigma,
\end{eqnarray*}
where
\[
s(\sigma) := \sum_{\tau \leq \sigma} t(\tau) \qquad \text{ for } \sigma \in \Pi_n.
\]
Therefore, using the orthogonal idempotence of the $\delta_\sigma$s, we have
\begin{eqnarray*}
T^2 &=&  \sum_{\sigma \in \Pi_n} s(\sigma)^2 \delta_\sigma \\
&=& \sum_{\sigma \in \Pi_n} \Big( s(\sigma)^2 \sum_{\tau \geq \sigma} \mu(\sigma, \tau) \tau \Big) \\
&=& \sum_{\tau \in \Pi_n} \Big( \sum_{\sigma \leq \tau} \mu(\sigma, \tau) s(\sigma)^2 \Big) \tau
\end{eqnarray*}
It follows from \eqref{eq:an} that
\begin{eqnarray*}
a_n &=& \sum_{\sigma \in \Pi_n} \mu(\sigma, \widehat{1})  \, s(\sigma)^2 \\
&=& \sum_{\sigma \in \Pi_n} (-1)^{\ell(\sigma)} (\ell(\sigma)-1)!  \, s(\sigma)^2,
\end{eqnarray*} 
using the facts that the interval $[\sigma, \widehat{1}]$ in the partition lattice $\Pi_n$ is isomorphic to the smaller partition lattice $\Pi_{\ell(\sigma)}$ -- because the coarsenings of $\sigma$ are obtained by arbitrarily merging blocks of $\sigma$ -- and the M\"obius number of the partition lattice $\Pi_k$ is $\mu_{\Pi_k}(\widehat{0}, \widehat{1}) = (-1)^{k-1} (k-1)!$. 
\end{proof}

Using Proposition \ref{prop:summands}, one easily computes by hand the first values of the sequence:
\[
a_1 = 1, \quad
a_2 = 3, \quad
a_3 = 39, \quad
a_4 = 1242.
\]

\section{Future directions}

\begin{enumerate}
\item
Find other simplicial polytopes with an elegant combinatorial structure that have the harmonic polytope as a Minkowski summand. 
\item
Use 1. to discover and explore other combinatorial models for Lagrangian geometry of matroids. 
Section \ref{sec:matroids} explains that the bipermutahedron is one such polytope, and leads to a theory of Lagrangian combinatorics of matroids, which is the subject of \cite{ADH2}. Other answers to 1. will lead to other such theories, and give rise to interesting matroid-theoretic directions.
\item
Study the Ehrhart polynomial and $h^*$-polynomial of $H_{n,n}$.
\item
Find a triangulation or subdivision of $H_{n,n}$ that will shed light on 3. In particular, $H_{n,n}$ is a Minkowski sum of one simplex and $n(n-1)$ segments, and its mixed subdivisions are likely to have a rich combinatorial structure.
\end{enumerate}

\section{Acknowledgments}

The first author would like to thank Graham Denham and June Huh for the very rewarding collaboration that led to the construction of the harmonic polytope. We would like to thank the anonymous referees for their careful reading of the work and their valuable suggestions to improve the exposition.

\bibliographystyle{habbrv}
\bibliography{ref}

\begin{thebibliography}{1}

\bibitem{AHK}
Karim Adiprasito, June Huh, and Eric Katz.
\newblock Hodge theory for combinatorial geometries.
\newblock {\em Ann. of Math. (2)}, 188(2):381--452, 2018.

\bibitem{A}
Federico Ardila.
\newblock The geometry of matroids.
\newblock {\em Notices Amer. Math. Soc.}, 65(8):902--908, 2018.

\bibitem{biperm}
Federico Ardila,
\newblock The bipermutahedron.
\newblock  arXiv:2008.02295.

\bibitem{ADH1}
Federico Ardila, Graham Denham, and June Huh. 
\newblock Lagrangian geometry of matroids.
\newblock arXiv:2004.13116.

\bibitem{ADH2}
Federico Ardila, Graham Denham, and June Huh. 
Lagrangian combinatorics of matroids.
In preparation, 2021.

\bibitem{AK}
Federico Ardila and Caroline~J. Klivans.
\newblock The {B}ergman complex of a matroid and phylogenetic trees.
\newblock {\em J. Combin. Theory Ser. B}, 96(1):38--49, 2006.

\bibitem{difficult-volume}
Imre B\'{a}r\'{a}ny and Zolt\'{a}n F\"{u}redi.
\newblock Computing the volume is difficult.
\newblock {\em Discrete Comput. Geom.}, 2(4):319--326, 1987.

\bibitem{Berstein}
D.~N. Bernstein.
\newblock The number of roots of a system of equations.
\newblock {\em Funkcional. Anal. i Prilo\v{z}en.}, 9(3):1--4, 1975.


\bibitem{Billera}
Louis J. Billera, The algebra of continuous piecewise polynomials, Adv. Math. 76 (1989), no. 2, 170--183.

\bibitem{Brion}
Michel Brion, Piecewise polynomial functions, convex polytopes and enumerative geometry, Parameter spaces
(Warsaw, 1994), Banach Center Publ., vol. 36, Polish Acad. Sci. Inst. Math., Warsaw, 1996, pp. 25--44.



\bibitem{Brylawski}
Thomas Brylawski.
\newblock The {T}utte polynomial part {I}: {G}eneral theory.
\newblock In {\em Matroid theory and its applications}, volume~83 of {\em
  C.I.M.E. Summer Sch.}, pages 125--275. Springer, Heidelberg, 2010.



\bibitem{Cox-Little-Schenck}
David~A. Cox, John~B. Little, and Henry~K. Schenck.
\newblock {\em Toric varieties}, volume 124 of {\em Graduate Studies in
  Mathematics}.
\newblock American Mathematical Society, Providence, RI, 2011.

\bibitem{Dawson}
Jeremy~E. Dawson.
\newblock A collection of sets related to the {T}utte polynomial of a matroid.
\newblock In {\em Graph theory, {S}ingapore 1983}, volume 1073 of {\em Lecture
  Notes in Math.}, pages 193--204. Springer, Berlin, 1984.

\bibitem{Dyer}
M.~E. Dyer.
\newblock The complexity of vertex enumeration methods.
\newblock {\em Math. Oper. Res.}, 8(3):381--402, 1983.

  
\bibitem{Dyer-Frieze}
M.~E. Dyer and A.~M. Frieze.
\newblock On the complexity of computing the volume of a polyhedron.
\newblock {\em SIAM J. Comput.}, 17(5):967--974, 1988.


\bibitem{FY}
Eva~Maria Feichtner and Sergey Yuzvinsky.
\newblock Chow rings of toric varieties defined by atomic lattices.
\newblock {\em Invent. Math.}, 155(3):515--536, 2004.

\bibitem{Herzog-Hibi-Ohsugi}
J\"{u}rgen Herzog, Takayuki Hibi, and Hidefumi Ohsugi.
\newblock {\em Binomial ideals}, volume 279 of {\em Graduate Texts in
  Mathematics}.
\newblock Springer, Cham, 2018.

\bibitem{HK}
June Huh and Eric Katz.
\newblock Log-concavity of characteristic polynomials and the {B}ergman fan of
  matroids.
\newblock {\em Math. Ann.}, 354(3):1103--1116, 2012.

\bibitem{Huh}
June Huh.
\newblock Milnor numbers of projective hypersurfaces and the chromatic
  polynomial of graphs.
\newblock {\em J. Amer. Math. Soc.}, 25(3):907--927, 2012.

\bibitem{Huh2}
June Huh.
\newblock {$h$}-vectors of matroids and logarithmic concavity.
\newblock {\em Adv. Math.}, 270:49--59, 2015.

\bibitem{Linial}
Nathan Linial.
\newblock Hard enumeration problems in geometry and combinatorics.
\newblock {\em SIAM J. Algebraic Discrete Methods}, 7(2):331--335, 1986.

\bibitem{LRS}
Luc\'{\i}a L\'{o}pez~de Medrano, Felipe Rinc\'{o}n, and Kristin Shaw.
\newblock Chern-{S}chwartz-{M}ac{P}herson cycles of matroids.
\newblock {\em Proc. Lond. Math. Soc. (3)}, 120(1):1--27, 2020.


\bibitem{Postnikov}
Alexander Postnikov.
\newblock Permutohedra, associahedra, and beyond.
\newblock {\em Int. Math. Res. Not. IMRN}, (6):1026--1106, 2009.

\bibitem{EC1}
Richard~P. Stanley.
\newblock {\em Enumerative combinatorics. {V}olume 1}, volume~49 of {\em
  Cambridge Studies in Advanced Mathematics}.
\newblock Cambridge University Press, Cambridge, second edition, 2012.

\bibitem{MacPherson}
R.~D. MacPherson.
\newblock Chern classes for singular algebraic varieties.
\newblock {\em Ann. of Math. (2)}, 100:423--432, 1974.

\bibitem{McMullen}
P.~McMullen.
\newblock Valuations and {E}uler-type relations on certain classes of convex
  polytopes.
\newblock {\em Proc. London Math. Soc. (3)}, 35(1):113--135, 1977.

\bibitem{Sabbah}
C.~Sabbah.
\newblock Quelques remarques sur la g\'{e}om\'{e}trie des espaces conormaux.
\newblock Number 130, pages 161--192. 1985.
\newblock Differential systems and singularities (Luminy, 1983).

\bibitem{Schwartz}
Marie-H\'{e}l\`ene Schwartz.
\newblock Classes caract\'{e}ristiques d\'{e}finies par une stratification
  d'une vari\'{e}t\'{e} analytique complexe. {I}.
\newblock {\em C. R. Acad. Sci. Paris}, 260:3262--3264, 1965.

\bibitem{SVV}
Aron Simis, Wolmer~V. Vasconcelos, and Rafael~H. Villarreal.
\newblock On the ideal theory of graphs.
\newblock {\em J. Algebra}, 167(2):389--416, 1994.

\bibitem{Sturmfels}
Bernd Sturmfels.
\newblock {\em Gr\"{o}bner bases and convex polytopes}, volume~8 of {\em
  University Lecture Series}.
\newblock American Mathematical Society, Providence, RI, 1996.

\bibitem{Villarreal}
Rafael~H. Villarreal.
\newblock Rees algebras of edge ideals.
\newblock {\em Comm. Algebra}, 23(9):3513--3524, 1995.

\bibitem{Ziegler}
G\"{u}nter~M. Ziegler.
\newblock {\em Lectures on polytopes}, volume 152 of {\em Graduate Texts in
  Mathematics}.
\newblock Springer-Verlag, New York, 1995.



\end{thebibliography}

\end{document}